\RequirePackage[l2tabu, orthodox]{nag}

\documentclass[reqno]{amsart}

\usepackage{lmodern}
\usepackage[T1]{fontenc}
\usepackage[utf8]{inputenc}
\usepackage[english]{babel}
\usepackage{microtype} 
\usepackage{comment}
\usepackage{varwidth} 
\usepackage{xcolor}

\usepackage{amsmath,amssymb,amsthm,mathrsfs,listings,xskak,
latexsym,mathtools,mathdots,booktabs,enumitem,bm,url}
\newcommand{\blackrook}{ \raisebox{-0.4em}{\BlackRookOnWhite} } 
\usepackage[centertableaux]{ytableau}
\usepackage[pdftex]{hyperref}
\hypersetup{
    colorlinks=true,
    linkcolor={blue!80},
    citecolor={black},
    urlcolor={blue!90}
}

\usepackage[capitalize,noabbrev]{cleveref}
\usepackage{caption}
\usepackage{subcaption}

\usepackage{todonotes}
\presetkeys%
    {todonotes}%
    {inline,backgroundcolor=yellow}{}

\newtheorem{theorem}{Theorem}[section]
\newtheorem{proposition}[theorem]{Proposition}
\newtheorem{lemma}[theorem]{Lemma}
\newtheorem{corollary}[theorem]{Corollary}

\usepackage[margin=2.8cm]{geometry}

\theoremstyle{definition}
\newtheorem{definition}[theorem]{Definition}
\newtheorem{example}[theorem]{Example}
\newtheorem{remark}[theorem]{Remark}

\newtheorem{fact}[theorem]{Fact}

\definecolor{lightblue}{rgb}{0.8,0.8,1.0}
\definecolor{lightgreen}{rgb}{0.8,1.0,0.8}
\definecolor{pBlue}{RGB}{86,139,190}
\definecolor{pCyan}{RGB}{149,186,201}
\definecolor{pSand}{RGB}{184,166,121}
\definecolor{pAlgae}{RGB}{87,115,135}
\definecolor{pSkin}{RGB}{236,216,167}
\definecolor{pGray}{RGB}{156,175,156}
\definecolor{pPink}{RGB}{215,114,127}
\definecolor{pOrange}{RGB}{211,153,80}

  \tikzset{
dot/.style = {circle, fill, minimum size=#1,
              inner sep=0pt, outer sep=0pt},
dot/.default = 6pt 
}

\newcommand{\defin}[1]{%
\relax\ifmmode%
\textcolor{blue}{#1}%
\else\textcolor{blue}{\emph{#1}}%
\fi%
}

\newcommand{\thsup}{\textnormal{th}}

\renewcommand{\rook}{\raisebox{-0.1em}{\symrook}}

\tikzset{every picture/.append
	style={
		scale=1,
		x=1em,
		y=1em,
		entries/.style={xshift=-0.5em,yshift=-0.5em,font=\small},
		thickLine/.style={line width=1.4pt,line join=round},
		bgEntry/.style={xshift=-0.5em,yshift=-0.5em,
			regular polygon,regular polygon sides=4,fill,inner sep=0pt,minimum size=1.35em
		}
	}
}
\usetikzlibrary{calc}

\DeclareMathOperator{\rk}{rk}

\newcommand{\NN}{\text{NN}}

\newcommand{\rookMat}{\mathcal{R}}

\makeatletter
\newcommand{\preceqdot}{\mathrel{\mathpalette\pr@ceqd@t\relax}}
\newcommand{\pr@ceqd@t}[2]{%
  \begingroup
  \sbox\z@{$#1\prec$}\sbox\tw@{$#1\preceq$}%
  \dimen@=\dimexpr\ht\tw@-\ht\z@\relax
  {\preceq}%
  \mkern-5mu
  \raisebox{\dimen@}{$\m@th#1\cdot$}%
  \endgroup
}
\makeatother

\title{Positroidal aspects of non-nesting rook placements}

\author[Aryaman Jal]{Aryaman Jal}
\address{Department of Mathematics, KTH Royal Institute of Technology
SE-100 44 Stockholm, Sweden}
\email{aryaman@kth.se}


\begin{document}

\begin{abstract}
    Rook matroids were recently introduced by the author and Alexandersson as matroids whose bases arise from certain restricted rook placements on a skew-shaped board. They were shown to be a subclass of transversal matroids and positroids. We further investigate the structural properties of rook matroids with an emphasis on the positroidal point of view. In particular, we  characterize rook matroids in terms of Grassmann necklaces of positroids, answering a question of Lam (2024). Along the way, we give a new proof of the positroidal structure of rook matroids and determine an important subclass of their cyclic flats.   
\end{abstract}

\maketitle 

\section{Introduction}

\noindent Matroids represent abstractions of the phenomena of linear independence and acyclicity. They can be constructed from a wide range of mathematical objects; we will be interested in ones arising in a combinatorial setting. For example, each of graphs, set systems, and lattice paths respectively give rise to the classes of graphic matroids, transversal matroids~\cite{EdmondsFulkerson1965TransversalMatroids}, and lattice path matroids~\cite{Bonin2003lattice, Bonin2006lattice}. A new entrant on this list is the family of rook matroids --- defined in terms of certain restricted rook placements on skew-shaped boards --- that were introduced by the author and Alexandersson in~\cite{alexandersson2024rooks}. In \cite{alexandersson2024rooks}, the structural properties of the rook matroid were studied and in particular, rook matroids were shown to be transversal matroids and positroids. In this paper, we delve deeper into the positroidal structure of rook matroids. 

Positroids were introduced by Blum in~\cite{Blum01BaseSortableMatroids} and rediscovered by Postnikov in~\cite{Postnikov2006TotalPositivity} who uncovered their connection to the theory of total positivity. This connection has since been mined in myriad directions, including algebraic geometry~\cite{KnutsonLamSpeyer2013PositroidJuggling}, statistical mechanics~\cite{GalashinPylyavskyy2020Ising}, and knot theory~\cite{Galashin2020PositroidsKnotsCatalan, GalashinLam2024PlabicLinks}. A  burgeoning area at the interface of geometry and theoretical physics is  that of scattering amplitudes; combinatorial objects related to positroids play a key role in this theory~\cite{ArkaniHamedLamBai2017PositiveGeometries, Arkani2016GrassmannianGeometryScatteringAmplitude, KarpWilliamsZhangDecompositionsAmplituhedra}. 

Across these areas, the usefulness of positroids arises from their multi-faceted combinatorial structure. Well-known classes of combinatorial objects parametrize positroids; these include Grassmann necklaces, bounded affine permutations, Le-diagrams, and more~\cite{Postnikov2006TotalPositivity, KnutsonLamSpeyer2013PositroidJuggling}. Recently, a new axiomatization of positroids called essential sets was introduced by Mohammadi and Zaffalon~\cite{Mohammadi2024EssentialSets}. We treat these in the context of rook matroids.  

Lattice path matroids are also known to be positroids~\cite{Oh2011Positroids}. While the approach in the aforementioned paper was direct, this fact can be deduced from older results (see \cite{Blum01BaseSortableMatroids, LamPostnikov2024Polypositroids} for example) due to the interval presentation (as a transversal matroid) of lattice path matroids. In contrast, such an approach is not possible for rook matroids and instead requires us to explicitly determine the Grassmann necklaces in a rook-theoretic way. In~\cite{alexandersson2024rooks}, we completely characterized when a rook matroid and lattice path matroid on the same skew shape are isomorphic. We also showed that despite being non-isomorphic in general, a lattice path matroid and rook matroid on the same skew shape nevertheless have the same Tutte polynomial. It is therefore of interest to further disambiguate these classes of matroids; our approach is from the positroidal perspective. In this paper, we study rook matroids from a positroidal point of view and in particular establish the following. \begin{enumerate}
    \item If the set of non-nesting rook placements on a board $B$ forms a matroid, then $B$ must be a skew shape, thereby proving a converse of \cite[Theorem 3.3]{alexandersson2024rooks}.
    \item Rook matroids are sort-closed matroids. This provides a new proof of the positroidal structure of rook matroids and gives a rook-theoretic interpretation of the two sorting operators that arise in the theory of alcoved polytopes.
    \item The essential sets of a rook matroid $\rookMat_{\lambda/\mu}$ are entirely determined by the corners of $\lambda/\mu$. In particular, the inequalities of the matroid polytope of $\rookMat_{\lambda/\mu}$ can be read off from the corners of the skew shape alone, a result that can be seen in parallel to the lattice path matroid case~\cite{KnauerMartinezSandovalRamirez}. 
    \item We characterize rook matroids in terms of Grassmann necklaces, thereby answering a question of Lam~\cite{Lam2024PrivateCommunication}. 
\end{enumerate}

\section{Preliminaries}

\noindent In this section, we collect the main definitions of the objects we study in the rest of this paper. 

A \defin{board} $B$ with $r$ rows and $c$ columns is a subset of the rectangular grid $[r] \times [c]$. 
We label the rows with elements from $[r]$ and the columns with elements from $[r+1, r+c]$ respectively, as in Figure~\ref{fig:rho_row_col}. Elements of $B$ are called \defin{cells}.
The row labeling is done from the top to bottom and the column labeling is done from left to right. We refer to cells $(i, j)$ of $B$ with respect to this labeling until otherwise specified. Given a board $B$ with $r$ rows and $c$ columns, the bipartite graph corresponding to $B$ is the graph $\Gamma_{B} = ([1, r]\sqcup [r+1,r+c]), E$ where $i \in [1,r]$ is connected to $j \in [r+1, r+c]$ if $(i, j) \in B$. A board $B$ is \defin{connected} if the bipartite graph $\Gamma_{B}$ is connected. Throughout, we will assume that all boards have the property that each row and each column has at least one cell.  

We will mainly be interested in boards arising from Ferrers and skew Ferrers shapes. An \defin{integer partition} $\lambda = (\lambda_{1}, \ldots , \lambda_{n})$ is any non-increasing tuple of positive integers. Following the English convention, the \defin{Ferrers shape} of $\lambda$ is a left-justified, weakly decreasing array of boxes, with $\lambda_{i}$ cells in the $i^\thsup$ row. For the purpose of brevity, we will frequently omit the commas when writing partitions. 

Given two integer partitions $\lambda \supseteq \mu$ of length at most $r$,
we let the \defin{skew Ferrers board} associated with $\lambda/\mu$ be the 
board with cells
\[
  \left\{ (i, r+j) : 1 \leq i \leq r \text{ and } \mu_i < r + j \leq \lambda_i \right\}.
\]
When $\mu$ is the empty partition, we treat $\mu$ as the all-zeroes vector. In what follows, we also refer to a skew Ferrers board as a \defin{skew shape}. An \defin{outer corner} of $\lambda/\mu$ is a cell $(i, j) \notin \lambda/\mu$ such that $(i, j-1), (i-1, j) \in \lambda/\mu$. Similarly, an \defin{inner corner} of $\lambda/\mu$ is a cell $(i, j) \notin \lambda/\mu$ such that $(i+1, j), (i, j+1) \in \lambda/\mu$. The \defin{size} of a skew shape $\lambda/\mu$ is the number of boxes in its diagram and is denoted $|\lambda/\mu|$. 

A \defin{non-attacking} placement of rooks on a board is a configuration of rooks where no two rooks
can take each other in the sense of chess. More formally, a non-attacking rook 
placement $\rho$ on a board $B$ is a subset of $B$ no two cells of which share a row or column index. Two rooks in a non-attacking rook placement form a \defin{nesting} if 
one rook lies South-East of another. Hereafter, a \defin{non-nesting rook placement} 
(or just rook placement, when it is clear that the setting is a non-nesting one) on a 
board $B$ is a non-attacking rook placement such that no pair of rooks forms a nesting. We use $\NN_{B}$ to mean the collection of non-nesting rook placements on $B$.

\section{Positroidal aspects of rook matroids}
\noindent Matroids can be axiomatized in various ways. We will identify a matroid by its collection of bases. For undefined terminology we refer the reader to~\cite{Oxley2006MatroidTheory}. 
\begin{definition}\label{defn:matroid_defn}
    A matroid $M$ is given by a  pair $(E, \mathcal{B})$ where $E$ is a finite set and $\mathcal{B}$ is a collection of subsets of $E$ satisfying the following two properties: \begin{enumerate}
        \item $\mathcal{B} \neq \emptyset$.
        \item For each distinct pair $B_{1}, B_{2}$ in $\mathcal{B}$ and $a \in B_{1} \setminus B_{2}$, there exists some $b \in B_{2} \setminus B_{1}$ such that~$(B\setminus \{a\}) \cup \{b\} \in \mathcal{B}$.
    \end{enumerate}
\end{definition}

In \cite{alexandersson2024rooks}, it was shown that a matroid structure can be placed on the set of non-nesting rook placements on a skew Ferrers diagram in the following way. Hereafter, any usage of the word non-nesting for a rook placement implicitly assumes that the rooks are in a non-attacking configuration. Given a non-nesting rook placement $\rho$ on a board $B$ we identify $\rho$ with the set $R(\rho) \cup C(\rho)$, where $R(\rho)$ is the set of row indices occupied by $\rho$ and $C(\rho)$ is the set of column indices that are not occupied by $\rho$. See Figure~\ref{fig:rho_row_col} for an example. Apart from being a more economical way of describing non-nesting rook placements, this description also lays bare the following matroid structure.

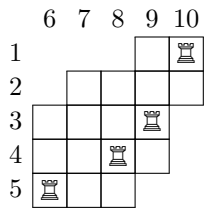
\begin{figure}[h]
    \centering
    \begin{tikzpicture}[inner sep=0in,outer sep=0in]
\node (n) {\begin{varwidth}{6cm}{
\ytableausetup{boxsize=1.25em}
\begin{ytableau} \none & \none[6] & \none[7] & \none[8] & \none[9] & \none[10]  \\ \none[1] & \none & \none & \none &  & \rook \\ \none[2] & \none &  & & & \\ \none[3] &  &  &  & \rook  \\ \none[4] &  & & \rook & \\ \none[5] &  \rook & & \\ \end{ytableau}}\end{varwidth}};
\end{tikzpicture}

    \caption{The rook placement $\rho$ on $55443/31$ is identified with set representation $R(\rho) \cup C(\rho) = \{1,3,4,5,7\}$.}
    \label{fig:rho_row_col}
\end{figure}

\begin{theorem}\cite[Theorem 3.3]{alexandersson2024rooks}
    Let $\lambda / \mu$ be a skew Ferrers shape with $r$ rows and $c$ columns. The \defin{rook matroid} on $\lambda / \mu$ is the matroid with groundset $[r+c]$ and collection of bases of the form: \[
    \{R(\rho) \cup C(\rho): \rho \in \mathrm{NN}_{\lambda / \mu}\}
    \]
    where $\mathrm{NN}_{\lambda / \mu}$ denotes the set of non-nesting rook placements on $\lambda / \mu$.
\end{theorem}

Well-known classes of matroids can be realized as rook matroids. When the skew shape is a straight shape $\lambda$, the rook matroid is a Schubert matroid (also known as generalized Catalan matroid). 

\begin{example}\label{example1:rectangles}
Let $n \geq 2$ and $1 \leq k \leq n-1$. If $\lambda/\mu = (n-k)^{k}$, an $(n-k) \times k$ rectangular board, then the set of non-nesting rook placements  $\NN_{\lambda/\mu}$ is of size $\binom{n}{k}$~\cite[Example 2.3]{alexandersson2024rooks} and the resulting matroid is the uniform matroid $U_{k, n}$ of rank $k$ on $n$ elements. 
\end{example}

In \cite{alexandersson2024rooks}, the connection between rook matroids and another well-studied class of matroids defined on a skew shape --- namely lattice path matroids \cite{Bonin2003lattice, Bonin2006lattice} --- was studied from the structural point of view. More specifically, it was shown  that if $\lambda/\mu$ is a $332/1$-avoiding skew shape, then the rook matroid $\mathcal{R}_{\lambda/\mu}$ is isomorphic to the lattice path matroid $\mathcal{P}_{\lambda/\mu}$; see \cite[Section 3.2]{alexandersson2024rooks} for more details. The tight correspondence between these matroids is one explanation for why there exists boards that are not skew shapes that do not support the structure of a rook matroid \cite[Remark 3.4]{alexandersson2024rooks}. 

The next proposition gives a self-contained proof of a stronger fact, characterizing skew shapes as the \emph{only} boards on which the set of non-nesting rook placements forms a matroid. In the following we assume that the rows of a board $B$ are labeled with elements from $[r]$ and the columns of $B$ are labeled with elements from $[r+1, r+c]$. We also use the following fact, the proof of which is straightforward and hence omitted. 

\begin{fact}\label{fact:skew_shape_characterization}
    A board $B$ is a skew shape if and only if for every pair of cells $(i, j), (k, \ell) \in B$ with $i<k$ and $j< \ell$, we also have $(k, j), (i, \ell) \in B$.
\end{fact}

\begin{proposition}\label{prop:skew_shape_necessary}
    Let $B$ be a board with $r$ rows and $c$ columns, and let \[
    \mathcal{B} = \{R(\sigma) \cup C(\sigma): \sigma \in \mathrm{NN}_{B}\},
    \]
    where $\mathrm{NN_{B}}$ denotes the set of non-nesting rook placements on $B$. If $M = ([r+c], \mathcal{B})$ is a matroid with bases $\mathcal{B}$, then the board $B$ must be equal to some skew shape. 
\end{proposition}

\begin{proof}
We begin by noting that if $i$ is a row index, then the bases of  $M\setminus i$ correspond to the non-nesting rook placements on $B$ such that row $i$ is not occupied. Similarly, if $j$ is a column index, then the bases of $M/j$ correspond to the non-nesting rook placements on $B$ such that column $j$ is not occupied. 

Now suppose $B$ is not a skew shaped board. Then, by Fact~\ref{fact:skew_shape_characterization}, there exists $(i, j), (k, \ell) \in B$ with $i<k$ and $j<\ell$ such that $(k, j)$ or $(i, \ell)$ do not lie in $B$. Let $M_{1} = (M\setminus X)/Y$ where $X = [r]\setminus \{i, k\}$ and $Y = [r+1, r+c] \setminus \{j, \ell\}$. Let $M_{2}$ be the matroid obtained by deleting all loops and coloops of $M_{1}$. Then, by the observation in the first paragraph of this proof, $M_{2}$ is isomorphic to the matroid $M_{3} = ([4], \mathcal{B}')$ where \[
\mathcal{B}' = \{R(\sigma) \cup C(\sigma): \sigma \in \mathrm{NN}_{B'}\},
\]
and $B'$ is one of the three boards in Figure~\ref{fig:boards_are_skew}. In each case, $\{2,3\}, \{1,4\} \in \mathcal{B}'$ and $3 \in \{2,3\}\setminus \{1,4\}$ but neither of $\{1,2\}$ nor $\{2,4\}$ are valid non-nesting rook placements on any of these boards. 
\end{proof}

\begin{remark}
    A similar strategy can be used to prove a lattice path matroid analogue of the above result. That is, if $B$ is a board and $L_{B}$ is the set of (East steps of) lattice paths beginning at the South-Westernmost corner of $B$ and terminating at the North-Easternmost corner of $B$ such that $L_{B}$ forms the bases of matroid, then $B$ is a skew shaped board.   
\end{remark}

\begin{figure}[!ht]
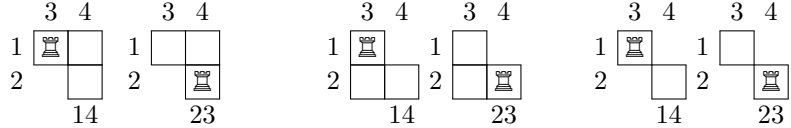

    \centering
    \begin{ytableau} \none & \none[3] & \none[4]  \\ \none[1] & \rook & & \none  \\ \none[2] & \none  & \\ \none & \none & \none[14]  \end{ytableau}
    \hspace{-0.5cm}
  \begin{ytableau} \none & \none[3] & \none[4]  \\ \none[1] &  & & \none  \\ \none[2] & \none  & \rook  \\ \none & \none & \none[23]  \end{ytableau}
\qquad
 \begin{ytableau} \none & \none[3] & \none[4]  \\ \none[1] & \rook  & \none  \\ \none[2] &  &  \\ \none & \none & \none[14]
   \end{ytableau}
   \hspace{-0.25cm}
\begin{ytableau} \none & \none[3] & \none[4]  \\ \none[1] &  & \none  \\ \none[2] &  & \rook  \\ \none & \none & \none[23]  \end{ytableau}
\qquad
\begin{ytableau} \none & \none[3] & \none[4]  \\ \none[1] & \rook  & \none  \\ \none[2] & \none  &  \\ \none & \none & \none[14]
   \end{ytableau}
   \hspace{-0.25cm}
\begin{ytableau} \none & \none[3] & \none[4]  \\ \none[1] &  & \none  \\ \none[2] & \none  & \rook  \\ \none & \none & \none[23]  \end{ytableau}

\caption{For each of the three boards $B'$ above, the collection $\mathcal{B}'$ in the proof of Proposition~\ref{prop:skew_shape_necessary} fails to satisfy the basis-exchange axiom.}
\label{fig:boards_are_skew}
\end{figure}

In what follows, we will be concerned with the positroidal properties of rook matroids. Postnikov defined a number of combinatorial objects parametrizing positroids~\cite{Postnikov2006TotalPositivity}; we will work with Grassmann necklaces. 

\begin{definition}
    Let $k \leq n$ be positive integers. A Grassmann necklace of type $(k, n)$ is a sequence $(I_{1}, \ldots , I_{n})$ of $k$-subsets $I_{i} \in \binom{[n]}{k}$ such that for any $i \in [n]$, \begin{enumerate}
        \item If $i \in I_{i}$, then $I_{i+1} = I_{i} \setminus \{i\} \cup \{j\}$ for some $j \in [n]$.
        \item If $i \notin I_{i}$, then $I_{i+1} = I_{i}$, where $I_{n+1} \coloneqq  I_{1}$.
    \end{enumerate}
\end{definition}

Postnikov showed that Grassmann necklaces give rise to positroids, and vice-versa~\cite{Postnikov2006TotalPositivity}. In the case of rook matroids, the Grassmann necklace has a nice description in terms of certain ``extremal'' rook placements~\cite{alexandersson2024rooks} that we now describe. 

We call such rook placements $i$-extremal because, informally speaking, they are formed by placing rooks as far to the right as possible, in each row of the skew shape. The Grassmann necklaces of rook matroids can be expressed in terms of this set of rook placements.

\begin{figure}[ht]
    \centering
    \begin{subfigure}{0.40\textwidth}
\centering
   \begin{tikzpicture}[inner sep=0in,outer sep=0in]
\node (n) {\begin{varwidth}{6cm}{
\ytableausetup{boxsize=1.25em}
\begin{ytableau} \none & \none[6] & \none[7] & \none[8] & \none[9] & \none[10]  \\ \none[1] & \none & \none & \none &  & \\ \none[2] & \none &  & & & \\ \none[3] &  &  &  & \rook  \\ \none[4] &  & & \rook & \\ \none[5] & & \rook & \\ \end{ytableau}}\end{varwidth}};
\end{tikzpicture}
\end{subfigure}
~
\begin{subfigure}{0.40\textwidth}
\centering
\begin{tikzpicture}[inner sep=0in,outer sep=0in]
\node (n) {\begin{varwidth}{6cm}{
\ytableausetup{boxsize=1.25em}
\begin{ytableau} \none & \none[6] & \none[7] & \none[8] & \none[9] & \none[10]  \\ \none[1] & \none & \none & \none &  & \\ \none[2] & \none & \rook  & & & \\ \none[3] & \rook  &  &  &  \\ \none[4] &  & & & \\ \none[5] & & & \\
\end{ytableau}}\end{varwidth}};
\end{tikzpicture}

\end{subfigure}

    \caption{The $3$-extremal and $8$-extremal non-nesting rook placement on $54421/31$ on the left and right respectively.}
    \label{fig:i_extremal_rook}
\end{figure}
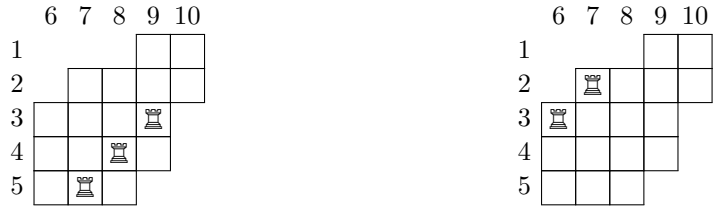

\begin{definition}\cite[Definition 3.33]{alexandersson2024rooks}\label{def:i_extremal_rook_placements} Let $\lambda / \mu$ be a skew shape with $r$ rows and $c$ columns. For $i \in [r+c]$, the \defin{$i$-extremal non-nesting rook placement} on $\lambda / \mu$ is the rook placement obtained by the following procedure.

If $i \in [1, r]$:
\begin{enumerate}
    \item Place a rook in the last cell of row $i$.
    \item For $j \in [i+1, r]$, place a rook in row $j$ as far to the right as possible, while ensuring the non-nesting condition holds. Continue until a rook is placed in the last row $r$ or first column $r+1$. 
\end{enumerate}

If $i \in [r+1, r+c]$:
\begin{enumerate}
    \item Leave column $i$ unoccupied.
    \item For $j \in [i+1, r+c]$, leave column $j$ unoccupied.
    \item For $j \in [1, r]$, place a rook in row $j$ as far to the right as possible while maintaining the unoccupied status of the columns in steps $(1)$ and $(2)$ and the non-nesting, non-attacking nature of the rook configuration.
\end{enumerate}
\end{definition}

The resulting rook placement is non-nesting since the non-nesting condition is maintained at every step. Examples are illustrated in Figure~\ref{fig:i_extremal_rook}.

\begin{proposition}\cite[Proposition 3.34]{alexandersson2024rooks}\label{prop:grassmann_of_rooks}
Let $\rookMat_{\lambda / \mu}$ be the rook matroid on skew shape $\lambda / \mu$ with $r$ rows and $c$ columns. Let $\mathcal{I} = (I_{1}, \ldots , I_{r+c})$ be the associated Grassmann necklace. Then for each $i \in [r+c]$, $I_{i}$ is the $i$-extremal non-nesting rook placement on $\lambda / \mu$.
\end{proposition}

We will use consequences of this construction repeatedly in the following subsections. To ease readability, we make the following definition and corollary. 

\begin{definition}\label{def:ell_ab_definition}
    Given a column index $a$ and row index $b$, we define \defin{$\ell_{a, b}$} to be the size of the set $I_{a}\cap [1,b]$: \[
    \ell_{a, b} \coloneqq |I_{a} \cap [1, b]|.
    \]
Similarly, given a row index $a$ and column index $b$, we define \defin{$m_{a, b}$} to be the size of the set $I_{a}\cap [r+1,b]$: \[
    m_{a, b} \coloneqq |I_{a} \cap [r+1, b]|.
    \]
\end{definition}

In words, $\ell_{a, b}$ equals the number of rooks of $\rho_{a}$ lying in the rows $1$ to $b$, while $m_{a, b}$ equals the number of unoccupied columns of $\rho_{a}$ in the range of columns $[r+1, b]$. When $a$ is a column index, the first rook of $\rho_{a}$ occurs in column $a-1$, and the following corollary holds. 

\begin{corollary}\label{cor:ell_ab_characterization}
Given a column index $a$ and row index $b$, we have $\ell_{a, b} =  0 \implies (b, a-1) \notin \lambda/\mu$. 
\end{corollary}

In the subsections that follow we will use the Grassmann necklace perspective to move between the rank conditions on subsets of the rook matroid and the rook placements themselves.

\subsection{Rook matroids are positroids: a new proof}
In \cite{alexandersson2024rooks}, Proposition~\ref{prop:grassmann_of_rooks} together with Oh's characterization of positroids~\cite{Oh2011Positroids} was used to establish the fact that every rook matroid is a positroid. While the case of lattice path matroids can be seen straightforwardly (by \cite[Theorem 5.2]{Blum01BaseSortableMatroids} or \cite[Theorem 4.6]{LamPostnikov2024Polypositroids}), that of rook matroids required more work. In this subsection, we furnish a new proof of this fact below, using instead the concept of sort-closed matroids, a class of matroids that is known to coincide with positroids, by work of Lam and Postnikov \cite[Corollary 9.4]{LamPostnikov2024Polypositroids}. The advantage of this is a more transparent proof of the positroid structure and an explicit interpretation of the sorting operators (defined below) in terms of rook placements. We first recall the definition of a sort-closed collection of sets. 

\begin{definition}\label{defn:sort_closed}
Let $I, J \in \binom{[n]}{k}$. Consider $I \cup J$ as the ordered multiset \[I \cup J = \{a_{1}\leq b_{1}\leq a_{2}\leq b_{2} \leq \ldots  a_{k}\leq b_{k}\}.\] Define two new subsets $\text{sort}_{1}(I, J)$ and $\text{sort}_{2}(I, J)$ by $\text{sort}_{1}(I, J) =  \{a_{1}, \ldots a_{k}\}$ and $\text{sort}_{2}(I, J) =  \{b_{1}, \ldots b_{k}\}$. Given a collection of subsets $\mathcal{F}\subset \binom{[n]}{k}$, we say that $\mathcal{F}$ is sort-closed if $I, J \in \mathcal{F}$ implies $\text{sort}_{1}(I, J), \, \text{sort}_{2}(I, J) \in \mathcal{F}$.
\end{definition}

A matroid $M = ([n], \mathcal{B})$ is said to be \textit{sort-closed }if its collection of bases $\mathcal{B}$ is sort-closed. The broader context where sort-closed collections naturally arise is that of alcoved polytopes, see~\cite{LamPostnikov07AlcovedPolytopesOne,BraunSolus2018rStableHypersimplices}. 

\begin{definition}\label{definition:some_rook_definitions}
We will use the following terminology concerning rook placements.
\begin{enumerate}
    \item A \defin{double rook placement} is a placement on $\lambda/\mu$ of two sets of rooks  --- distinguished by different colors--- on $\lambda/\mu$. No restriction is made on either set to be in a non-attacking configuration.
    \item Given two non-nesting rook placements $\rho_{1}$ and $\rho_{2}$ on $\lambda/\mu$, the collection $\rho_{1} \cup \rho_{2}$ is a \defin{double non-nesting rook placement}. Two rooks of different colors might share the same square or appear in the same row or same column. The only nestings will be between rooks of different colors.
    \item A pair of rooks $(i, j)$ and $(k, \ell)$ is \defin{strictly nested}, if one rook lies strictly South-East of another, that is $i<k$ and $j<\ell$.
     \item We \defin{uncross} a strictly nested pair of rooks by replacing the above pair by $(k, j)$ and $(i, \ell)$ respectively.
\end{enumerate}
See Figures~\ref{fig:double_rook2} and \ref{fig:double_rook1} for a double rook placement and double non-nesting rook placement respectively. In Figure~\ref{fig:double_rook1}, the pair of rooks in $(2, 8)$ and $(5, 10)$ is strictly nested, while in Figure~\ref{fig:double_rook11} this pair has been uncrossed. 
\end{definition}

We first describe how to take a double non-nesting rook placement and convert it to a double rook placement with special properties. We consider each rook placement $\rho_{1}$ and $\rho_{2}$ below as representing distinct colors.

\begin{definition}\label{definiton:uncrossed_double_rook}
Let $\rho_{1}\cup \rho_{2}$ be a double non-nesting rook placement on $\lambda/\mu$. Consider the following procedure: \begin{enumerate}
    \item For each row $i$ of $\lambda/\mu$, scan row $i$ of $\lambda/\mu$ from right to left. For each
    rook $s = (i, j)$ encountered, if $s$ strictly nests a rook  of a \emph{different} color, then uncross $s, t$ where $t = (k, \ell)$ and $k>i$ is the first row index after $i$ containing a rook that is nested by $s$. Repeat until there are no strictly nested rooks of different colors. 
    \item Perform step (1) again, this time only uncrossing any pairs of strictly nested rooks of the \emph{same} color that might have formed as a result of step (1). 
\end{enumerate} 
Let \emph{$\mathrm{uncross}(\rho_{1}, \rho_{2})$} be the resulting double rook placement on $\lambda/\mu$. The output of this operation is uniquely defined.
\end{definition}

That $\text{uncross}(\rho_{1}, \rho_{2})$  indeed lies on $\lambda/\mu$ arises from the fact that the board is a skew shape, and hence by Fact~\ref{fact:skew_shape_characterization}, every uncrossing of strictly nested rooks moves the strictly nested rooks to a pair of cells that lie on the board. We assign numbers to the rooks of $Z$ in the following manner.

\begin{definition}\label{definition:numbered_uncrossing}
    If $Z = \text{uncross}(\rho_{1}, \rho_{2})$ has $n$ rooks, the \defin{numbered uncrossing $Y$} is an assignment of the numbers $1$ to $n$ to the rooks of $Z$ via the following procedure: starting with the first row, scan each row of $\lambda/\mu$ from right to left and number the rooks of $Z$ that are encountered in this order. If two rooks share the same square assign the lower number to the black rook. 
\end{definition}

An example of the uncrossing procedure is illustrated in Figure~\ref{fig:uncrossing_example_rook}. Note that after step (1), we still have a strictly nested pair of rooks of the same color in cells $(3,8)$ and $(4,9)$, but after step (2), $Z$ has no strictly nested rooks. 

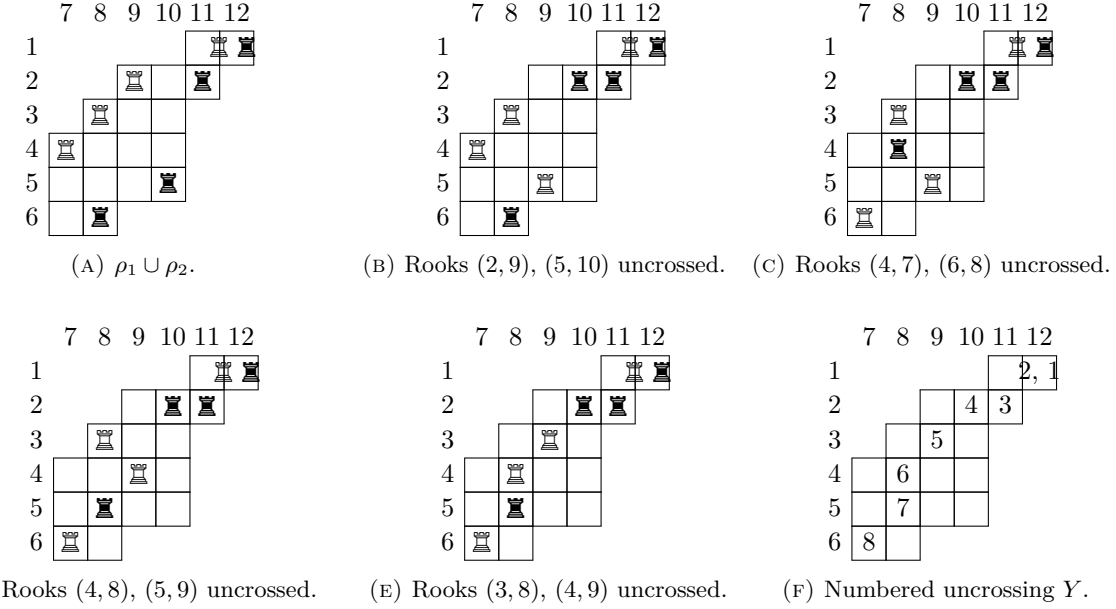
\begin{figure}[ht!]
    \centering
    \begin{subfigure}{0.30\textwidth}
\centering
   \begin{tikzpicture}[inner sep=0in,outer sep=0in]
\node (n) {\begin{varwidth}{6cm}{
\ytableausetup{boxsize=1.25em}
 \begin{ytableau} \none & \none[7] & \none[8] & \none[9] & \none[10] & \none[11] & \none[12]  \\ \none[1] & \none & \none & \none & \none &  & \rook \scalebox{0.5}{\blackrook}  \\ \none[2] & \none & \none & \rook  &  & \scalebox{0.5}{\blackrook}  \\ \none[3] & \none & \rook &  &  \\ \none[4] & \rook  &  &  &  \\ \none[5] &  &  &  &  \scalebox{0.5}{\blackrook} \\ \none[6] &  & \scalebox{0.5}{\blackrook} \\ \end{ytableau}}\end{varwidth}};
\end{tikzpicture}
\caption{$\rho_{1}\cup \rho_{2}$.}
    \label{fig:double_rook1}
\end{subfigure}
~
\hspace{0.2 cm}
\begin{subfigure}{0.30\textwidth}
\centering
\begin{tikzpicture}[inner sep=0in,outer sep=0in]
\node (n) {\begin{varwidth}{6cm}{
\ytableausetup{boxsize=1.25em}
 \begin{ytableau} \none & \none[7] & \none[8] & \none[9] & \none[10] & \none[11] & \none[12]  \\ \none[1] & \none & \none & \none & \none &  & \rook \scalebox{0.5}{\blackrook}  \\ \none[2] & \none & \none &  & \scalebox{0.5}{\blackrook} & \scalebox{0.5}{\blackrook}  \\ \none[3] & \none & \rook &  &  \\ \none[4] & \rook  &  &  &  \\ \none[5] &  &  & \rook &  \\ \none[6] &  & \scalebox{0.5}{\blackrook} \\ \end{ytableau}
}\end{varwidth}};
\end{tikzpicture}
\caption{Rooks $(2, 9)$, $(5, 10)$ uncrossed.}
    \label{fig:double_rook11}
\end{subfigure}
~
\begin{subfigure}{0.30\textwidth}
\centering
\begin{tikzpicture}[inner sep=0in,outer sep=0in]
\node (n) {\begin{varwidth}{6cm}{
\ytableausetup{boxsize=1.25em}
 \begin{ytableau} \none & \none[7] & \none[8] & \none[9] & \none[10] & \none[11] & \none[12]  \\ \none[1] & \none & \none & \none & \none &  & \rook \scalebox{0.5}{\blackrook}  \\ \none[2] & \none & \none &  & \scalebox{0.5}{\blackrook} & \scalebox{0.5}{\blackrook}  \\ \none[3] & \none & \rook &  &  \\ \none[4] &   & \scalebox{0.5}{\blackrook}   &  &  \\ \none[5] &  &  & \rook &  \\ \none[6] & \rook & \\ \end{ytableau}
}\end{varwidth}};
\end{tikzpicture}
\caption{Rooks $(4, 7)$, $(6, 8)$ uncrossed.}
    \label{fig:double_rook12}
\end{subfigure}
~
\\
\vspace{0.5cm}
\begin{subfigure}{0.30\textwidth}
\centering
\begin{tikzpicture}[inner sep=0in,outer sep=0in]
\node (n) {\begin{varwidth}{6cm}{
\ytableausetup{boxsize=1.25em}
 \begin{ytableau} \none & \none[7] & \none[8] & \none[9] & \none[10] & \none[11] & \none[12]  \\ \none[1] & \none & \none & \none & \none &  & \rook \scalebox{0.5}{\blackrook}  \\ \none[2] & \none & \none &  & \scalebox{0.5}{\blackrook}  & \scalebox{0.5}{\blackrook}  \\ \none[3] & \none & \rook &  &  \\ \none[4] &  &   & \rook &  \\ \none[5] &  & \scalebox{0.5}{\blackrook} &   &  \\ \none[6] & \rook  &  \\ \end{ytableau}}\end{varwidth}};
\end{tikzpicture}
\caption{Rooks $(4, 8)$, $(5, 9)$ uncrossed.}
    \label{fig:double_rook2}
\end{subfigure}
~
\hspace{0.2 cm}
\begin{subfigure}{0.30\textwidth}
\centering
\begin{tikzpicture}[inner sep=0in,outer sep=0in]
\node (n) {\begin{varwidth}{6cm}{
\ytableausetup{boxsize=1.25em}
 \begin{ytableau} \none & \none[7] & \none[8] & \none[9] & \none[10] & \none[11] & \none[12]  \\ \none[1] & \none & \none & \none & \none &  & \rook \scalebox{0.5}{\blackrook}  \\ \none[2] & \none & \none &  & \scalebox{0.5}{\blackrook}  & \scalebox{0.5}{\blackrook}  \\ \none[3] & \none &  & \rook &  \\ \none[4] &  & \rook   & &  \\ \none[5] &  & \scalebox{0.5}{\blackrook} &   &  \\ \none[6] & \rook  &  \\ \end{ytableau}}\end{varwidth}};
\end{tikzpicture}
\caption{Rooks $(3,8)$, $(4,9)$ uncrossed.}
    \label{fig:double_rook3}
\end{subfigure}
~
\begin{subfigure}{0.30\textwidth}
\centering
\begin{tikzpicture}[inner sep=0in,outer sep=0in]
\node (n) {\begin{varwidth}{6cm}{
\ytableausetup{boxsize=1.25em}
 \begin{ytableau} \none & \none[7] & \none[8] & \none[9] & \none[10] & \none[11] & \none[12]  \\ \none[1] & \none & \none & \none & \none &  & \tiny $2, 1$ \\ \none[2] & \none & \none &  & $4$ & $3$  \\ \none[3] & \none &  & $5$ &  \\ \none[4] &  & $6$  &  &  \\ \none[5] &  & $7$ &  &  \\ \none[6] & $8$  &  \\ \end{ytableau}}\end{varwidth}};
\end{tikzpicture}
\caption{Numbered uncrossing $Y$.}
    \label{fig:double_rook4}
\end{subfigure}
    \caption{On skew shape $654442/421$: (A) a double non-nesting rook placement $\rho_{1}\cup \rho_{2}$; (B) - (D) double rook placements obtained after uncrossing a strictly nested pair of rooks in the preceding figure; (E) $Z = \text{uncross}(\rho_{1}, \rho_{2})$ obtained after step (2) of uncrossing procedure; (F) Numbering rooks of $Z$.}
    \label{fig:uncrossing_example_rook}
\end{figure}

The double rook placement $Z = \text{uncross}(\rho_{1}, \rho_{2})$ might have two rooks of the \emph{same} color in the same row, or same column, but not in the same square. Different colored rooks might share the same square. We gather some consequences of the uncrossing procedure in the following lemma. 

\begin{lemma}\label{lem:properties_of_uncross}
Let $\rho_{1}, \rho_{2}$ be non-nesting rook placements on $\lambda/\mu$ represented by sets $I$ and $J$ respectively, and let $Z = \mathrm{uncross}(\rho_{1}, \rho_{2})$.  
\begin{enumerate}
    \item The element $i$ is a row index of $I \cup J$ appearing once (resp. twice) if and only if $Z$ has one (resp. two) rooks in row $i$.
    \item The element $j$ is a column index of $I \cup J$ appearing once (resp. twice) if and only if $Z$ has one (resp. no) rook in column $i$. Also, $j \notin I \cup J$ if and only if $Z$ has two rooks in column $j$.  
    \item The double rook placement $Z$ contains no pair of strictly nested rooks. 
    \item If $x$ is a rook in the numbered uncrossing $Y$ of $Z$, then the rook $x+1$ occurs to left of $x$ in the same row, below $x$ in the same column, or strictly South-West of $x$. 
\end{enumerate}
\end{lemma} 

\begin{proof}
First observe that at each step in the procedure to obtain $Z$ from $\rho_{1} \cup \rho_{2}$, the occupancy of rows and non-occupancy of columns of $\lambda/\mu$ is not changed. It thus suffices to prove properties (1) and (2) for the double non-nesting rook placement $\rho_{1} \cup \rho_{2}$. For (1), by definition of $I, J$ as bases of a rook matroid, if $i$ is a row index, then an occurrence of $i$ in $I \cup J$ corresponds to either $\rho_{1}$ or $\rho_{2}$ or both having a rook in row $i$. For (2), let $j \in I\cup J$ be a column index appearing once. This implies that one rook placement has no rook in column $j$, while the other does. If $j$ appears twice then neither of $\rho_{1}$ nor $\rho_{2}$ have a rook in row $j$. Similarly, if $j \notin I \cup J$, then each of $\rho_{1}$ and $\rho_{2}$ must have a rook in column $j$. For (3), note that the first step in producing $Z$ leaves no pair of strictly nested rooks of different colors, while the second step leaves no strictly nested rooks of the same colors. Finally for (4), note that this follows from the absence of strictly nested rooks in $Z$, by (3). 
\end{proof}

We are now ready to prove the main result of this subsection. 

\begin{theorem}\label{prop:rooks_sort_closed}
    Let $\lambda/\mu$ be a skew shape and let $\rookMat_{\lambda / \mu}$ be the corresponding rook matroid. Then $\rookMat_{\lambda / \mu}$ is a sort-closed matroid. In particular, $\rookMat_{\lambda /\mu}$ is a positroid. 
\end{theorem}

\begin{proof}
    Let $I, J$ be bases of $\rookMat_{\lambda/\mu}$ with corresponding non-nesting rook placements $\rho_{1}, \rho_{2}$ respectively. Let $\rho_{1}$ and $\rho_{2}$ be identified with the colors white and black respectively. Consider the double rook placement $Z = \mathrm{uncross}(\rho_{1}, \rho_{2})$, and the numbered uncrossing $Y$ corresponding to $Z$.  Let $\rho$ be the rook placement on $\lambda/\mu$ consisting of the odd-indexed rooks of $Y$ and let $\sigma$ is the rook placement consisting of the even-indexed rooks of $Y$. We will denote these by $\rho = \text{odd}(Z)$ and $\sigma = \text{even}(Z)$ respectively. See Figure~\ref{fig:sort_closed_procedure} for an example of the double non-nesting rook placement $\rho_{1} \cup \rho_{2}$, the double rook placement $Z$, its numbered counterpart $Y$, and the non-nesting rook placement $\rho = \text{odd}(Z)$. We will be done if we can prove the following claim.

\textbf{Claim:} Both $\rho$ and $\sigma$ are non-nesting rook placements on $\lambda/\mu$. They satisfy
\begin{align}
\text{sort}_{1}(I, J) &= R(\rho) \cup C(\rho),\label{eq:sort_1}\\
\text{sort}_{2}(I, J) &= R(\sigma) \cup C(\sigma) \label{eq:sort_2}.
\end{align}

First, observe that $\rho$ is indeed a non-attacking rook placement on $\lambda/\mu$: this follows from the fact that $\rho$ consists of the odd-numbered rooks of $Z$ and that each row and each column of $Z$ contains at most two rooks. The non-nesting nature of $\rho$  follows from the previous sentence together with Lemma~\ref{lem:properties_of_uncross} (3), which tells us that $Z$ contains no pair of strictly nested rooks. The same argument holds for $\sigma$. 

Regarding the set equalities in the claim, we will prove only equation~(\ref{eq:sort_1}), since (\ref{eq:sort_2}) follows by an analogous argument, considering even-indexed rooks of $Z$ instead of odd. To this end, it suffices to show that every row index in $R(\rho)$ and column index in $ C(\rho)$ corresponds to the odd-indexed elements in the sorted multiset $I \cup J$. Begin by noting that the double rook placement $Z$ represents the sorted multiset $I \cup J$. Suppose that the skew shape has $c$ columns and that\[
I \cup J = \{a_{1}\leq b_{1}\leq a_{2}\leq b_{2} \leq \cdots \leq a_{c}\leq b_{c}\}, \quad \text{and} \quad \text{sort}_{1}(I, J) = \{a_{1}\leq \cdots \leq a_{c}\}.
\]
Let the occupied row indices and unoccupied column indices of $\rho$ respectively be $ R(\rho) = \{t_{1}< \ldots < t_{\ell}\}$ and $C(\rho) = \{t_{\ell+1}< \ldots < t_{c}\}$. Let $i \in [\ell]$ and consider $a_{i}$, a row index of $I \cup J$. Then, by Lemma~\ref{lem:properties_of_uncross}~(1), $Z$ has a rook in row $a_{i}$ that has index, say $x$, in the numbered uncrossing $Y$. Since $a_{i}$ has odd index in the sorted multiset $I \cup J$, $x$ must be odd and hence $\rho$ has a rook in row $a_{i}$. By the same argument, we can show that since $\rho = \text{odd}(Z)$, every row occupied by a rook of $\rho$ arises from an odd-indexed element in the multiset $I \cup J$. This shows that $a_{i} = t_{i}$ for $i=1, \ldots , \ell$. 

Now suppose $j \in [\ell +1, c]$. We will show that $a_{j} =  t_{j}$ by induction. For the base case, consider the column index $a_{\ell +1}$; we need to show that $a_{\ell+1} \in C(\rho)$. Suppose $a_{\ell+1}$ occurs twice in $I\cup J$. By Lemma~\ref{lem:properties_of_uncross}~(2), $Z$ has no rook in column $a_{\ell+1}$, and since $\rho$ is formed by taking a subset of the rooks in $Z$, neither does $\rho$. Now suppose $a_{\ell+1}$ occurs exactly once. By Lemma~\ref{lem:properties_of_uncross}~(2), $Z$ has exactly one rook in column $a_{\ell+1}$. We need to show that the index of that rook in $Y$ is even. Two cases arise, either (a) $b_{\ell}$ is a row index, or (b) $b_{\ell}$ is a column index. 

\textbf{Case (a):} Here, if $b_{\ell}$ is a row index, then $a_{\ell+1}$ is the leftmost column index of $Z$ containing exactly one rook, say $s$. Let the rook in row $b_{\ell}$ be $t$. Then $t$ must have even index in $Y$. Now, if $t$ occurs in column $a_{\ell +1}$, then $s = t$ and hence $a_{\ell+1} \in C(\rho)$ and we are done. This is shown in Figure~\ref{fig:double_rook3} ($a_{\ell} = 5, b_{\ell} = 6, a_{\ell +1} = 7$). Now, suppose the column index of $t$ is some  $q < a_{\ell+1}$. By definition of $a_{\ell +1}$, the column $q$ contains two rooks, the first of which occurs at $(a_{\ell}, q)$ and hence has odd index in $Y$. Now every column from $q$ to $a_{\ell+1}-1$ also has two rooks. By Lemma~\ref{lem:properties_of_uncross}~(4), in the numbered uncrossing $Y$, for each column in the aforementioned range, the larger number is even and the smaller one is odd, which implies that the rook in column $a_{\ell+1}$ has even index, and hence $a_{\ell+1} \in C(\rho)$ required.  

\textbf{Case (b):} Here, if $b_{\ell}$ is a column index, then $b_{\ell}$ is the leftmost column index containing a single rook $s'$ of $Z$. Also, $a_{\ell}$ is the largest row index containing a rook $t'$ of $Z$, occurring in column $q'$, say. The rook $t'$ is in row $a_{\ell}$ and this row contains only $t'$, so it must have odd index in $Y$. This together with the fact that since $b_{\ell}$ is the leftmost column index containing a single rook of $Z$ implies that each column from $q'$ to $b_{\ell}-1$ in the numbered uncrossing $Y$ contains two rooks of $Z$, with the higher indexed one being odd and the lower indexed one, even. Then the rook $s'$ in column $b_{\ell}$ has odd index in $Y$.  Then $a_{\ell +1}$ is the index of the next column after $b_{\ell}$ that contains exactly one rook of $Z$. Combining this with the fact that $s'$ has odd index in $Y$ implies that every column in the range $[b_{\ell}+1, a_{\ell+1}-1]$  --- if any --- in the numbered uncrossing $Y$ has exactly two rooks, the larger of which is even and the smaller of which is odd. (Here, again, we used Lemma~\ref{lem:properties_of_uncross}~(4).) This implies that the rook in column $a_{\ell+1}$ is even and hence $a_{\ell+1} \in C(\rho)$. This case is shown in Figure~\ref{fig:sort_closed_numbered} ($a_{\ell} = 5, b_{\ell}=8, a_{\ell+1} = 9$). 

\textbf{Induction step:} Now suppose $j>\ell+1$ and we have $a_{\ell+1} = t_{\ell+1}, \ldots , a_{j-1} = t_{j-1}$. Consider the column index $a_{j}$. Once again, if $a_{j}$ occurs twice in $I \cup J$, then $Z$ has no rook in column $a_{j}$, and hence $\rho$ no rook in column $a_{j}$ and we are done. Suppose $a_{j}$ occurs exactly once, then by Lemma~\ref{lem:properties_of_uncross} (2), $Z$ has exactly one rook in column $a_{j}$. Consider the set of columns of $Z$ with exactly one rook, and order them by index. Consider the two columns $a_{m}< b_{k}$ preceding $a_{j}$ in this order. By the induction hypothesis, $a_{m} \in C(\rho)$ and hence the rook in column $a_{m}$ has even index in $Y$. Columns $a_{m}$ and $b_{k}$ are separated by columns each of which contain two rooks, the higher index of which is odd, and the lower index of which is even, by Lemma~\ref{lem:properties_of_uncross}~(4). This implies that the rook in column $b_{k}$ has odd index in $Y$, and by a similar argument, it follows that the rook in column $a_{j}$ has even index and hence $a_{j} \in C(\rho)$, so $a_{j} = t_{j}$. The induction and hence the proof is then complete. 
\end{proof}

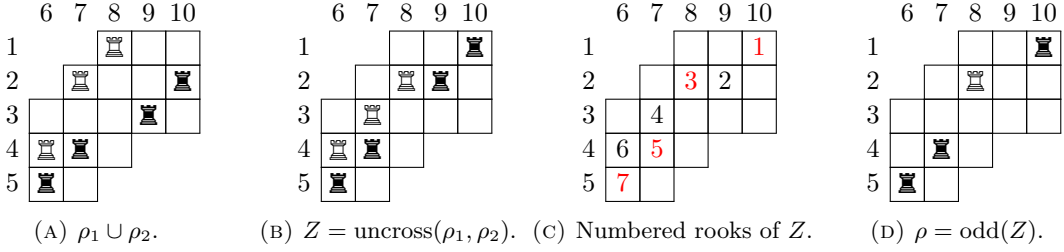
\begin{figure}[ht!]
    \centering
    \begin{subfigure}{0.25\textwidth}
\centering
   \begin{tikzpicture}[inner sep=0in,outer sep=0in]
\node (n) {\begin{varwidth}{6cm}{
\ytableausetup{boxsize=1.25em}
 \begin{ytableau} \none & \none[6] & \none[7] & \none[8] & \none[9] & \none[10]  \\ \none[1] & \none & \none & \rook  &  &  \\ \none[2] & \none & \rook &  &  & \scalebox{0.5}{\blackrook} \\ \none[3] &  &  &  & \scalebox{0.5}{\blackrook} &  \\ \none[4] & \rook  & \scalebox{0.5}{\blackrook} &  \\ \none[5] & \scalebox{0.5}{\blackrook} &  \\ \end{ytableau}}\end{varwidth}};
\end{tikzpicture}
\caption{$\rho_{1}\cup \rho_{2}$.}
    \label{fig:sort_closed_procedure1}
\end{subfigure}
\hspace{-0.35cm}
\begin{subfigure}{0.25\textwidth}
\centering
\begin{tikzpicture}[inner sep=0in,outer sep=0in]
\node (n) {\begin{varwidth}{6cm}{
\ytableausetup{boxsize=1.25em}
 \begin{ytableau} \none & \none[6] & \none[7] & \none[8] & \none[9] & \none[10]  \\ \none[1] & \none & \none &  &  & \scalebox{0.5}{\blackrook}  \\ \none[2] & \none &  & \rook & \scalebox{0.5}{\blackrook} &  \\ \none[3] &  & \rook &  &  &  \\ \none[4] & \rook & \scalebox{0.5}{\blackrook} &  \\ \none[5] & \scalebox{0.5}{\blackrook} &  \\ \end{ytableau}}\end{varwidth}};
\end{tikzpicture}
\caption{$Z = \mathrm{uncross}(\rho_{1}, \rho_{2})$.}
    \label{fig:sort_closed_procedure2}
\end{subfigure}
\hspace{-0.45cm}
\begin{subfigure}{0.25\textwidth}
\centering
\begin{tikzpicture}[inner sep=0in,outer sep=0in]
\node (n) {\begin{varwidth}{6cm}{
\ytableausetup{boxsize=1.25em}
 \begin{ytableau} \none & \none[6] & \none[7] & \none[8] & \none[9] & \none[10]  \\ \none[1] & \none & \none &  &  & \textcolor{red}{1}  \\ \none[2] & \none &  & \textcolor{red}{3}  & $2$ &  \\ \none[3] &  & $4$ &  &  &  \\ \none[4] & $6$ & \textcolor{red}{5}  &  \\ \none[5] & \textcolor{red}{7}  &  \\ \end{ytableau}}\end{varwidth}};
\end{tikzpicture}
\caption{Numbered rooks of $Z$.}
\label{fig:sort_closed_numbered}
\end{subfigure}
\hspace{-0.45cm}
\begin{subfigure}{0.25\textwidth}
\centering
\begin{tikzpicture}[inner sep=0in,outer sep=0in]
\node (n) {\begin{varwidth}{6cm}{
\ytableausetup{boxsize=1.25em}
 \begin{ytableau} \none & \none[6] & \none[7] & \none[8] & \none[9] & \none[10]  \\ \none[1] & \none & \none &  &  & \scalebox{0.5}{\blackrook}  \\ \none[2] & \none &  & \rook &  &  \\ \none[3] &  &  &  &  &  \\ \none[4] & & \scalebox{0.5}{\blackrook} &  \\ \none[5] & \scalebox{0.5}{\blackrook} &  \\ \end{ytableau}}\end{varwidth}};
\end{tikzpicture}
\caption{$\rho = \text{odd}(Z)$.}
    \label{fig:sort_closed_procedure3}
\end{subfigure}
    \caption{On skew shape $55532/21$: (a) a double non-nesting rook placement $I\cup J$ with $I = 124 9 10$ and $J = 23458$; (b) double rook placement $Z = \text{uncross}(I, J)$ with no strictly nested rooks; (c) numbered uncrossing $Y$ of $Z$ with odd indices in red; (d) non-nesting rook placement $\rho = \text{odd}(Z)$, which corresponds to $\text{sort}_{1}(I, J) = 12459$.}
    \label{fig:sort_closed_procedure}
\end{figure}

\subsection{Essential sets of a rook matroid}\label{subsection:essential sets}
We now consider a different family of objects that parametrizes positroids called essential sets. In \cite{Mohammadi2024EssentialSets}, Mohammadi and Zaffalon introduced \textit{ranked essential sets} of a positroid, in analogy with Fulton's essential sets of permutations, and initiated the systematic study of them; in particular a new axiomatization for positroids in terms of essential sets was derived together with an algorithm showing how to compute essential sets and their ranks for a given positroid.

The upshot of determining the ranked essential sets of a positroid is that the inequalities of the corresponding matroid polytope can be specified in terms of these objects. That is our aim with the next three results. 

Informally, essential sets are ``maximally dependent cyclic intervals,''  a term that will be made precise in the next proposition. By \cite{Mohammadi2024EssentialSets}, the following can be taken to be the definition of essential sets of a positroid. The cyclic interval $[i, j]$ is called an essential set and the collection of all pairs $(\rk([i, j]),  [i, j])$ is called the \emph{ranked essential family} of a positroid. 

\begin{proposition}\cite[Theorem 3.8]{Mohammadi2024EssentialSets}\label{prop:rank_essential_char}
    Let $M$ be a positroid of rank $k$ on $[n]$ and let $\mathcal{E}'$ be the set of pairs $(r, [i, j])$ such that: \begin{enumerate}
        \item If $i \neq j$, then $ r = \rk([i, j]) = \rk([i+1, j]) = \rk([i, j-1]) = \rk([i-1, j]) - 1 = \rk([i, j+1])-1$.
        \item If $i = j$, then $0 = r = \rk([i, i]) = \rk([i-1, i]) -1 = \rk([i, i+1]) - 1$.
    \end{enumerate}
Then the ranked essential family $\mathcal{E}$ of $M$ is given by $\mathcal{E} = \mathcal{E}' \cup \{k, [1, n]\}$.
\end{proposition}

Call an essential set \emph{non-degenerate} if it is not a singleton and \emph{proper} if it is not equal to $[1, n]$. The following definition identifies a subfamily of essential sets that is important in the context of being sufficient information from which the positroid can be recovered. 

\begin{definition}\cite[Definition 3.10]{Mohammadi2024EssentialSets}\label{def:connected_essential_definition}
    A ranked essential set $(r, [i, j])$ is \textit{connected} if there exists no pairwise disjoint sets $(r_{1}, [i_{1}, j_{1}]),$ $\ldots$, $(r_{m}, [i_{m}, j_{m}]) \in \mathcal{E}$ such that \[
    r = r_{1}+\ldots   r_{m}+ \left |[i, j]\setminus \bigcup_{a=1}^{m}[i_{a}, j_{a}] \right| \text{and $[i, j] \supseteq \bigcup_{a=1}^{m}[i_{a}, j_{a}]$}.
    \]
\end{definition}

We pause to clarify a distinction that went unremarked upon in~\cite{Mohammadi2024EssentialSets}. Namely, the above definition of a connected essential set coincides with the usual matroid theoretic definition of being connected, when the set under consideration is an essential set of a positroid. 

\begin{fact}\label{fact:connected_definitions_coincide}
    Let $(r, [i, j])$ be a ranked essential set of a positroid $M$. Then $(r, [i, j])$ is connected in the sense of Definition~\ref{def:connected_essential_definition} if and only if $[i, j]$ is connected as a subset of $M$, i.e., $M|_{[i, j]}$ has no separators. 
\end{fact}

Recall that if $M$ is a matroid on groundset $E$, then a \defin{separator} of $M$ is a subset $T$ of $E$ such that $\rk(M) = \rk(T) + \rk (E\setminus T)$. The ``if'' direction of Fact~\ref{fact:connected_definitions_coincide} follows immediately by the two  definitions of connectedness, whereas the ``only if'' direction requires more work. We omit the proof of Fact~\ref{fact:connected_definitions_coincide} since we will only be using its ``if'' direction and the characterization in Definition~\ref{def:connected_essential_definition}.

We first write down an explicit formula for the ranks of cyclic intervals of rook matroids in terms of extremal rook placements. As before, let $I_{a}$ be the $a^{\thsup}$ entry of the Grassmann necklace of $\rookMat_{\lambda/\mu}$ and let $\rho_{a}$ be the corresponding $a$-extremal non-nesting rook placement. Recall the definition of $\ell_{a, b}$ and $m_{a, b}$ from Definition~\ref{def:ell_ab_definition}. 

\begin{lemma}\label{lem:rank_cyclic_interval}
    Let $\lambda/\mu$ be a skew shape with $r$ rows and $c$ columns. \begin{enumerate}
        \item If $a$ is a column index and $b$ is a row index, then \[
      \rk([a, b]) = |I_{a} \cap [a, b]| = r+c-a+1+\ell_{a, b}. 
    \]
    \item If $a$ is a row index and $b$ is a column index, then \[
    \rk([a, b]) = |I_{a} \cap [a, b]| = t_{a}-a+1+m_{a, b},\]
where $t_{a}$ is the last row index containing a rook of $\rho_{a}$. 
\item If $a$ and $b$ are both column indices with $b<a$, then \[
\rk([a, b]) = c-a+b+1+y_{a, b},\]
where $y_{a, b}= |\{(i, j) \in \rho_{a}: j \in [b+1, a-1]\}|$. In words, $y_{a, b}$ equals the number of rooks of $\rho_{a}$ in the columns $[b+1, a-1]$. 
\end{enumerate}  

\begin{remark}
    Note that for row indices $a$, the last rook of $\rho_{a}$ will either occur in row $r$ or column $r+1$. In the latter case, $t_{a}$ might not equal $r$, and each row between $a+1$ and $r$ will consist only of a single box. This will imply $t_{1} = t_{2} = \cdots = t_{t_{1}}$. 
\end{remark}
\end{lemma}

\begin{proof}
For the first equality in (1) we note that $I = I_{a} \cap [a, b]$ is an independent set inside $[a, b]$, since $I_{a}$ is a basis of the rook matroid. Further, there is no independent set $I' \subset [a, b]$ of larger size. This is because $I \cap [1, b]$ are row indices of the rooks in $\rho_{a}$ and by definition of $a$-extremality of the rook placement, no further rooks can be added in the rows $1$ to $b$. For the second equality in (1),
this follows since $|I_{a} \cap [a, b]| = |I_{a}\cap [a, r+c]|+|I_{a}\cap [1, b]| = r+c-a+1 + \ell_{a, b}$. 

The equalities in (2) follow by an identical argument as the ones in (1). For $(3)$, consider the rook placement consisting only of rooks of $\rho_{a}$ between the columns $b+1$ and $a-1$. The corresponding set $I_{a} \cap [a, b]$ is an independent set contained in $[a, b]$; it is of maximum size because the addition of any other rook to the aforementioned rook placement would necessarily be added in a column index $j$ occurring before $b$ or after $a$, which is inadmissible. This independent set contains exactly $r+c-a+1+b-r$ unoccupied column indices and exactly $y_{a, b}$ occupied row indices, as required. 
\end{proof}

We will encounter two notions of connectedness in what follows and in preparation of that, we state two lemmas that have easy proofs. We omit the proof of the first, and include the proof of the second.

\begin{lemma}\label{lem:connected_matroid_connected_shape}
    The rook matroid $\rookMat_{\lambda/\mu}$ is connected if and only if $\lambda/\mu$ is a connected skew shape. 
\end{lemma}

The idea behind the proof of the lemma above is essentially the same as in the case of lattice path matroids; see~\cite[Theorem 3.6]{Bonin2003lattice} for the details of that case.

\begin{lemma}\label{lem:intervals_inner_corners_connected}
    Let $(i, j)$ be an inner corner (resp. outer corner) of connected skew shape $\lambda/\mu$. Then $[j+1, i]$ (resp. $[i, j-1]$) is a connected subset of the rook matroid $\rookMat_{\lambda/\mu}$. 
\end{lemma}

\begin{proof}
    We only do the inner corner case since the outer corner case follows by an analogous argument. Let $\lambda/\mu$ be a skew shape with $r$ rows and $c$ columns. Let $M = \rookMat_{\lambda/\mu}$. Then $M|_{[j+1, i]} = M\setminus ([i+1, r]\cup [r+1, j])$. Now, since $[i+1, r]$ is a set of rows of $\lambda/\mu$, by \cite[Lemma 3.7]{alexandersson2024rooks}, $M_{1} = M\setminus [i+1, r]$ is again a rook matroid on the skew shape consisting of $\alpha/\beta$ with empty columns $r+1, \ldots ,j$ appended in front, where \[
    \alpha/\beta = \{(k, \ell) \in \lambda/\mu: 1 \leq k \leq i, \,\, j+1 \leq \ell \leq r+c\}.
    \]
   This is because since $(i, j)$ is an inner corner, the columns $r+1$ to $j$ share no cells with $\alpha/\beta$. In particular, $[r+1, j]$ is a set of coloops of $M_{1}$ which implies that $M|_{[j+1, i]} = M_{1} \setminus [r+1, j] = M_{1}/[r+1, j] = \rookMat_{\alpha/\beta}$, where the last equality follows from~\cite[Lemma 3.7]{alexandersson2024rooks} again. Since $\alpha/\beta$ is the restriction of a connected skew shape $\lambda/\mu$ to a cyclic interval of indices, $\alpha/\beta $ is also a connected skew shape. Hence, by Lemma~\ref{lem:connected_matroid_connected_shape}, $M|_{[j+1, i]}$ is a connected matroid, which completes the proof. 
\end{proof}

The following lemma locates the endpoints of essential sets of a rook matroid in terms of corners of the underlying skew shape. 

\begin{lemma}\label{lem:connected_sets_endpoints_rook}
    Let $[a, b]$ be a non-degenerate proper essential set of a rook matroid $\rookMat_{\lambda/\mu}$. \begin{enumerate}
        \item If $a$ is a column index and $b$ is a row index, then there is an inner corner of $\lambda/\mu$ at $(b, a-1)$.
    
        \item If $a$ is a row index and $b$ is a column index, then there is an outer corner of $\lambda/\mu$ at $(a, b+1)$.

        \item If $a$ and $b$ are both column indices, then $\lambda/\mu$ has an inner corner in column $a-1$ and an outer corner in column $b+1$. 
    \end{enumerate} 
\end{lemma}

\begin{proof}
For (1), since $[a, b]$ is a non-degenerate essential set, the following rank equations hold: \begin{equation}\label{ell_ab_equations}
\rk([a, b]) = \rk([a+1, b]) = \rk([a, b-1]) = \rk([a-1, b]) - 1 = \rk([a, b+1])-1.
\end{equation}
From Lemma~\ref{lem:rank_cyclic_interval}~(1), the above equations can be rewritten as \[
\ell_{a, b} = \ell_{a-1, b} =  \ell_{a, b-1} = \ell_{a+1, b} -1 = \ell_{a, b+1} -1.  
\]

To show that $(b, a-1)$ is an inner corner of $\lambda/\mu$, we need to show that (i) $(b, a-1) \notin \lambda/\mu$, while (ii) $(b, a), (b+1, a-1) \in \lambda/\mu$. For (i), by Corollary~\ref{cor:ell_ab_characterization}, it is enough to show that $\ell_{a, b} = 0$. Suppose, to the contrary, that  $\ell_{a, b} \geq 1$. Since $\ell_{a, b} = \ell_{a, b-1}$, $|I_{a} \cap [1, b]| = |I_{a} \cap [1, b-1]|$. This implies, in particular, that $\rho_{a}$ can have no rooks in row $b$. Let $s = (m, k)$ be the last rook of $\rho_{a}$ in the interval $[1, b]$. \textcolor{black}{Then the first cell of row $b$ occurs at $(b, k)$: if it were to the left of column $k$, then $\rho_{a}$ would be able to accommodate a rook in row $b$, which would be a contradiction. For the same reason, $\ell_{a, b} = \ell_{a-1, b}$ implies that all the rooks of $\rho_{a-1}$ also occur in $[1, b-1]$. See~Figure~\ref{fig:inner_outer_corner_essential_set1} for an illustration of this idea.} By $(a-1)$-extremality, the last rook of $\rho_{a-1}$ occurs in $(m, k-1)$. Since the first rook of $\rho_{a-1}$ occurs in column $a-2$, the column index of the last rook of $\rho_{a-1}$ occurs strictly to the left of $a-1$. In other words, $k-1 < a-1$ and $m<b$, since all the rooks of $\rho_{a}$ occur in $[1, b-1]$. Then, $(m, k-1), (b, a-1) \in \lambda/\mu$ with $k-1 < a-1$ and $m<b$, which implies by Fact~\ref{fact:skew_shape_characterization}, that $(b, k-1) \in \lambda/\mu$, which contradicts the fact that the first cell in row $b$ occurs at $(b, k)$.

Thus $\ell_{a, b} = 0$ which, by Corollary~\ref{cor:ell_ab_characterization}, implies that $(b, a-1) \notin \lambda/\mu$. This proves (i). Since $\ell_{a, b} = 0$, from the equations above, it follows that $\ell_{a+1, b} = \ell_{a, b+1} = 1$. The fact that $\ell_{a, b+1} = 1$ and $\ell_{a, b} = 0$ implies that the first rook of $\rho_{a}$ occurs at $(b+1, a-1)$, and hence this cell lies in $\lambda/\mu$. Since $\ell_{a+1, b} =1$, $\rho_{a+1}$ has a rook at $(i, a)$ for some $i \in [1, b]$. Row $b$ must have some cell $(b, j)$ of $\lambda/\mu$ South-East of $(i, a)$; otherwise $\lambda/\mu$ contains an empty row at the index $b$ which contradicts the skew shape assumption. Then, applying Fact~\ref{fact:skew_shape_characterization} to $(i, a)$ and $(b, j)$, it follows that $(b, a)$ lies in $\lambda/\mu$, which proves (ii) in the above paragraph, and hence point (1) in the lemma.

For (2), we use a similar argument as in (1), this time using Lemma~\ref{lem:rank_cyclic_interval}~(2) to rewrite the rank equations of an essential set to get conditions on $m_{a, b}$ analogous to (\ref{ell_ab_equations}). Here $m_{a, b}$ can be non-zero. 

For (3), using Lemma~\ref{lem:rank_cyclic_interval}~(3), the rank equations can be rewritten as \[
y_{a, b} = y_{a-1, b} = y_{a, b+1} = y_{a+1, b} - 1 =  y_{a, b-1} -1. 
\]
Observe that by $a$-extremality of $\rho_{a}$, if $\rho_{a-1}$ has a rook in row $s$, say, then $\rho_{a}$ also has a rook in row $s$ that occurs weakly to the right of $\rho_{a-1}$. This --- together with the fact that $y_{a, b} =  y_{a-1, b}$ --- implies that $\rho_{a-1}$ and $\rho_{a}$ have their first rooks in the same row. 

Suppose, to the contrary, that $\lambda/\mu$ has no inner corner in column $a-1$. Then $\rho_{a}$ and $\rho_{a+1}$ have their first rooks in the same row of $\lambda/\mu$. Thus, by the last sentence of the previous paragraph, all three $\rho_{i}$ for $i \in \{a-1, a, a+1\}$ have their first rooks in the same row. Now, in the interval of columns $[b+1, a]$, $\rho_{a+1}$ has one rook more than $\rho_{a}$ and $\rho_{a-1}$. In this range of columns, let the last rooks of $\rho_{a-1}$ and $\rho_{a+1}$ respectively occur at $(i, b+1)$ and $(j, k)$ for some row indices $i<j$ and a column index $k \in [b+1, a]$. \textcolor{black}{If $k = b+1$, by $a$-extremality,  $\rho_{a}$ will then have a rook in $(j, k)$. This contradicts the fact that $y_{a, b} = y_{a-1, b}$. See~Figure~\ref{fig:inner_outer_corner_essential_set2} for an illustration of this idea.} Thus, $k>b+1$, which by Fact~\ref{fact:skew_shape_characterization} applied to $(i, b+1)$ and $(j, k)$, then implies the cell $(j, b+1) \in \lambda/\mu$, which means $\rho_{a}$ must have a rook on that cell, again contradicting the fact that $y_{a, b} = y_{a-1, b}$. Thus, $\lambda/\mu$ has an inner corner in column~$a-1$. 

Since $y_{a, b} =  y_{a, b+1}$, $\rho_{a}$ can have no rook in column $b+1$. Since $y_{a, b} = y_{a, b-1}-1$, $\rho_{a}$ must have a rook in column $b$, \textcolor{black}{which by $a$-extremality will occur in the last cell of the corresponding row.} These facts taken together imply that there is an outer corner in column $b+1$, finishing the proof of (3). 
\end{proof}

\begin{figure}[ht!]
    \centering
    \begin{subfigure}{0.35\textwidth}
\centering
   \begin{tikzpicture}[inner sep=0in,outer sep=0in]
\node (n) {\begin{varwidth}{6cm}{
\ytableausetup{boxsize=1.25em}
 \begin{ytableau} \none & \none[\scriptstyle{k-1}] & \none[\scriptstyle{k}] & \none[\scriptstyle{a-2}] & \none[\, \, \, \scriptstyle{a-1}] & \none[] & \none[] & \none[]  \\ \none[] & \none &  & \scalebox{0.5}{\blackrook}  & \rook & \\ \none[] & \none & \scalebox{0.5}{\blackrook}  & \rook &  &   \\ \none[\scriptstyle{m}] & \scalebox{0.5}{\blackrook}  & \rook & &  \\ \none[\scriptstyle{b}] & &  \\ \end{ytableau}}\end{varwidth}};
\end{tikzpicture}
\caption{$\rho_{a-1}$ in black, $\rho+{a}$ in white}
    \label{fig:inner_outer_corner_essential_set1}
\end{subfigure}
\hspace{1.5cm}
\begin{subfigure}{0.35\textwidth}
\centering
\begin{tikzpicture}[inner sep=0in,outer sep=0in]
\node (n) {\begin{varwidth}{6cm}{
\ytableausetup{boxsize=1.25em}
 \begin{ytableau} \none & \none[] & \none[] & \none[\scriptstyle{b+1}] & \none[\scriptstyle{k}] & \none[\scriptstyle{a-2}] & \none[\, \, \, \scriptstyle{a-1}] & \none[\scriptstyle{a}]  \\ \none[] & \none &  &  &  & \textcolor{red}{\rook} & \scalebox{0.5}{\blackrook} & \rook \\ \none[] & \none &  &  & \textcolor{red}{\rook} & \scalebox{0.5}{\blackrook} & \rook  \\ \none[i] &  &  & \textcolor{red}{\rook} & \scalebox{0.5}{\blackrook} & \rook \\ \none[j] &  & &  & \rook \\ \end{ytableau}
 }\end{varwidth}};
\end{tikzpicture}
\caption{$\rho_{a-1}, \rho_{a}, \rho_{a+1}$ in red, black, white respectively.}
    \label{fig:inner_outer_corner_essential_set2}
\end{subfigure}
    \caption{In the proof of Lemma~\ref{lem:connected_sets_endpoints_rook}: (A) In point (1), $\ell_{a, b} = \ell_{a-1, b}$ leads to a contradiction with there being no rooks of $\rho_{a}$ or $\rho_{a-1}$ in row $b$. (B) In point (3), $y_{a,b} = y_{a-1, b} =y_{a+1, b}-1$ leads to a contradiction to the fact that $\rho_{a}$ (in black) has no rook in column $b+1$.}
    \label{fig:inner_outer_corner_essential_set}
\end{figure}
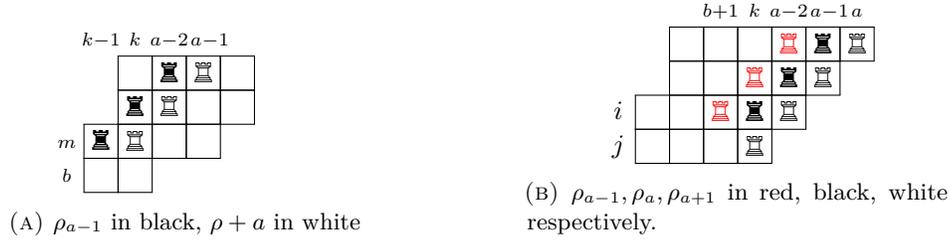

The following proposition allows us to compute the ranked essential sets of a rook matroid by simply looking at the inner (resp. outer) corners of the skew shape and reading off the number of columns to the right (resp. left). We restrict to the case of connected skew shapes for ease of presentation. 

\begin{proposition}\label{prop:ess_sets_rooks}
    Let $\lambda / \mu$ be a connected skew shape on $r$ rows and $c$ columns and $M = \rookMat_{\lambda / \mu}$ be the corresponding rook matroid on $\lambda / \mu$. Each pair in the family of connected ranked essential sets of the matroid $M$ is of the form: \begin{enumerate}
        \item $(r+c-j, [j+1, i])$ for every inner corner $(i, j)$ of $\lambda / \mu$, and 
        \item $(l-1-r, [k, l-1])$ for every outer corner $(k,l)$ of $\lambda / \mu$.
    \end{enumerate}
\end{proposition}

\begin{proof}
    We first show that every rank-interval pair of the above form satisfies the rank conditions of Proposition~\ref{prop:rank_essential_char}. We only do the verification for inner corners since the outer corner case follows analogously (using the second statement instead of the first in each of Lemma~\ref{lem:rank_cyclic_interval} and Lemma~\ref{lem:connected_sets_endpoints_rook}).

   Let $(i, j)$ be an inner corner of $\lambda/\mu$. We need to show that the following equations hold: \begin{align}
        \rk([j+1, i]) &= r+c-j \label{eq:rank1_1}\\ \rk([j+2, i]) &= r+c-j  \label{eq:rank1_2}\\
        \rk([j+1, i-1]) &= r+c-j  \label{eq:rank1_3}\\
        \rk([j, i]) &= r+c-j+1 \label{eq:rank1_4}\\
        \rk([j+1, i+1]) &= r+c-j+1
        \label{eq:rank1_5}
    \end{align} 

    Consider the cyclic interval $[j+1, i]$. This interval has rank $r+c-j$: it contains $[j+1, r+c]$, which being a selection of columns, is an independent set of size $r+c-j$. Further, if $A \subseteq [j+1, i]$ is an independent set of size $r+c-j+1$, by the augmentation axiom of independent sets, there exists some $a \in A$ such that $I = [j+1, r+c]\cup\{a\}$ is also independent. Since $I$ uses all the columns to the right of $j$, $a$ must necessarily be a row index (less than $i$, since $A \subseteq [j+1, i]$). But then the rook in row $a$ must correspond to some column $b \in [j+1, r+c]$, implying that $a, b$ both lie in $I$, contradicting the independence of $I$. Thus $[j+1, i]$ has rank $r+c-j.$ Thus equation~(\ref{eq:rank1_1}) holds.

    Since $(i, j)$ is an inner corner, $(i, j+1) \in \lambda / \mu$. In particular $[j+2, r+c] \cup \{i\}$ is an independent set of size $r+c-j$ inside $[j+2, i]$; this proves equation~(\ref{eq:rank1_2}). Equation~(\ref{eq:rank1_3}) holds since $[j+1, i-1]$ contains the independent set $[j+1, r+c]$. 
    
    Consider the cyclic interval $[j, i]$. This has rank $c-j+1$ for the following reason. The inequality $\rk([j, i]) \leq \rk([j+1, i]) +1$ gives the upper bound and  since $[j, i]$ contains the independent set $\{j, j+1, \ldots , r+c\}$ of size $r+c-j+1$, we also get the same lower bound. Note that the latter set is independent since every subset of columns is independent in a rook matroid. This proves equation~(\ref{eq:rank1_4}). 

    Consider the cyclic interval $[j+1, i+1]$. This also has rank $r+c-j+1$ since it contains the independent set $[j+1, r+c] \cup \{i+1\}$. Note that this set corresponds to the rook placement with a single rook in $(i+1, j)$, which can be made since $(i+1, j) \in \lambda / \mu$, which holds since $(i, j)$ is an inner corner. Thus equation~(\ref{eq:rank1_5}) holds.

    Finally each of the essential sets coming from the inner and outer corners of the skew shape are connected. Indeed, by (the ``if'' direction of ) Fact~\ref{fact:connected_definitions_coincide}, it is enough to show that each of these sets is connected in the matroid theoretic sense. This follows from Lemma~\ref{lem:intervals_inner_corners_connected}. (It is here that we use the assumption that $\lambda/\mu$ is a connected skew shape.)

We now show that every connected essential set arises from the corners of the skew shape. Let $[a, b]$ be a connected, non-degenerate essential set. Assume that $a$ is a column index; the case when $a$ is a row index follows by a similar argument. Now $b$ cannot be a column index greater than $a$ since otherwise $[a, b]$ would be an independent set (being a collection of columns) contradicting the fact that essential sets are dependent. Hence, two cases arise:
    
    (1) The element $b$ is a column index with $b<a$. By Lemma~\ref{lem:connected_sets_endpoints_rook}~(3), there exists row indices $i, j$ such that there is an inner corner at  $(i, a-1)$ and an outer corner in column $(j, b+1)$. By the first part of this proof, this means that $[a, i]$ and $[j, b]$ are (connected) essential sets of $M$. Since $a>b$, and $\lambda/\mu$ is a skew shape, it follows that $i<j$. We can then write the cyclic interval $[a, b]$ as \begin{align}
    [a, b] &= [a, i] \cup \cup_{k=i+1}^{j-1}[k, k] \cup [j, b], \quad \text{with} \label{eq:connected_essential_eqs1}\\ \rk([a, b]) &= \rk([a, i]) + \rk([j, b]) + \left |[i+1, j-1] \right|. \label{eq:connected_essential_eqs2}
    \end{align}
    The equality in the second line can be verified as follows. By Lemma~\ref{lem:rank_cyclic_interval}~(3), $a$ and $b$ are both column indices and hence $\rk[a, b] = c-a+b+1+y_{a, b}$, where $y_{a, b}$ equals the number of rooks of $\rho_{a}$ in the columns $[b+1, a-1]$. In this case, $y_{a, b}$ equals the number of rows in the interval $[i+1, j-1]$. Rearranging terms, we get \[
    \rk([a, b]) = (r+c-(a-1)+1)+(b-r)+j-1-i,
\]
which is precisely the right-hand side of (\ref{eq:connected_essential_eqs2}). Thus (\ref{eq:connected_essential_eqs1}) and (\ref{eq:connected_essential_eqs2}) gives us a contradiction to the fact that $[a, b]$ is a connected essential set.
    
    (2) The element $b$ is a row index. By Lemma \ref{lem:connected_sets_endpoints_rook} (1), it follows immediately that there is an inner corner of $\lambda/\mu$ at $(b, a-1)$. Also, by Lemma~\ref{lem:rank_cyclic_interval}, \[ 
    \rk([a,b]) = r+c-a+1+\ell_{a, b},
    \]
    and since $(b, a-1)$ is an inner corner, $\ell_{a, b} = 0$. Thus $(r+c-a+1, [a, b])$ is a ranked essential set such that $(b, a-1)$ is an inner corner of $\lambda/\mu$, finishing the proof.
\end{proof}

See Figure~\ref{fig:inner_outer_essential} for an example of the essential sets of a rook matroid arising from the corners of the underlying shape.

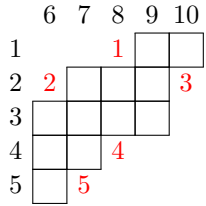
\begin{figure}[h]
    \centering
    \begin{tikzpicture}[inner sep=0in,outer sep=0in]
\node (n) {\begin{varwidth}{6cm}{
\ytableausetup{boxsize=1.25em}
\begin{ytableau} \none & \none[6] & \none[7] & \none[8] & \none[9] & \none[10]  \\ \none[1] & \none & \none & \none[$\textcolor{red}{$1$}$] &  & \\ \none[2] & \none[$\textcolor{red}{$2$}$] &  & & & \none[$\textcolor{red}{$3$}$] \\ \none[3] &  &  &  &  \\ \none[4] &  & & \none[$\textcolor{red}{$4$}$] \\ \none[5] & & \none[$\textcolor{red}{$5$}$]  \\ \end{ytableau}}\end{varwidth}};
\end{tikzpicture}

    \caption{$\lambda / \mu = 54421/31$ with corners numbered. The corresponding proper connected essential sets are: \textcolor{red}{(1)} $(2, [9, 1])$, \textcolor{red}{(2)} $(4, [7, 2])$, \textcolor{red}{(3)} $([4, [2, 9]])$, \textcolor{red}{(4)} $(2, [4, 7])$, and \textcolor{red}{(5)} $(1, [5, 6])$. In this case, together with $(5, [1,10])$, these are all the ranked essential sets of the positroid  $\rookMat_{\lambda/\mu}$.}
    \label{fig:inner_outer_essential}
\end{figure}

The following is an immediate consequence of Proposition~\ref{prop:ess_sets_rooks} and the fact that the connected essential sets of a positroid determine the facets of the corresponding matroid polytope~\cite[Corollary 3.12]{Mohammadi2024EssentialSets}. It can also be seen as a companion result to~\cite[Theorem 3.3]{KnauerMartinezSandovalRamirez}, where the authors obtain a facet description of the matroid polytope of a lattice path matroid only in terms of the combinatorics of the upper and lower bounding paths. 

\begin{proposition}
    Let $\lambda / \mu$ be a connected skew shape on $r$ rows and $c$ columns and $M = \rookMat_{\lambda / \mu}$ be the corresponding rook matroid on $\lambda / \mu$. Let $IC(\lambda/\mu)$ and $OC(\lambda / \mu)$ respectively be the sets of inner and outer corners of $\lambda / \mu$. Then the base  polytope of the rook matroid $P_{M}$ consists of all points $x \in \mathbb{R}^{r+c}$ such that \begin{align*}
        \sum_{i \in [r+c]}x_{i} &= c, \\
        \sum_{\ell \in [j+1, i]}x_{\ell} &\leq r+c-j \quad \text{for every $(i, j) \in IC(\lambda/\mu)$}, \\ 
        \sum_{\ell \in [i, j-1]}x_{\ell} & \leq j-1-r \quad  \text{for every $(i, j) \in OC(\lambda/\mu)$}. 
    \end{align*}
\end{proposition}

\subsection{A characterization theorem for Grassmann necklaces}\label{subsection:characterization_theorem}
We now turn our focus to determining when a given Grassmann necklace of a positroid corresponds to one from a rook matroid. We use the following convention when referring to sets in a $(k, n)$-Grassmann necklace $\mathcal{I} = (I_{1}, \ldots , I_{n})$ of a positroid. Set \begin{align*}
R_{i} &= I_{i} \cap [1, n-k], \, \, \,\qquad r_{i} = \min R_{i}, \qquad \qquad \qquad \, \text{for $i \neq n-k+1$,}\\
C_{i} &= I_{i} \cap [n-k+1, n], \quad c_{i} = \max ([n-k+1, n]\setminus C_{i}),\\
[n-k+1, n] \setminus C_{i} &= S_{i} = \{s_{\ell}<\ldots < s_{1}\}.
\end{align*}
Set $r_{i} = 0$ for $i\geq n+1$. We use the $R_{i}$ and $C_{i}$ notation to emphasize the roles played by row and column indices in a rook matroid. From the definitions of $R_{i}$ and $C_{i}$, it follows that $c_{i} \leq c_{i-1}$ for $i \in [2, n-k]$ and $r_{i+1} \leq r_{i}$ for $i \in [n-k+2, n]$. In the following lemma, we characterize when strict inequality holds in the case of rook matroids. 

\begin{lemma}\label{lem:outer_corners}
Let $\mathcal{I} = (I_{1}, \ldots , I_{r+c})$ be the Grassmann necklace of a rook matroid $\rookMat_{\lambda/\mu}$.  Suppose that $i \neq 1$ is a row index (resp. $i \neq r+c$ is a column index). Then the following are equivalent: \begin{enumerate}
    \item There is an outer corner (resp. inner corner) of $\lambda/\mu$ at $(i, c_{i}+1)$ (resp. at $(r_{i}-1, i-1)$). 
    \item We have $c_{i} < c_{i-1}$ (resp. $r_{i} > r_{i+1}$). 
\end{enumerate}
Each of the above also implies that $I_{i}$ is not a cyclic interval.
\end{lemma}

\begin{proof}
For the first part of the lemma, let $\rho_{i}$ be the non-nesting rook placement corresponding to $I_{i}$. We will only handle the case when $i$ is a row index, since the column index case follows by an analogous argument. We begin by noting that the rightmost cell of row $i$ in $\lambda/\mu$ is precisely equal to $c_{i}$ since it  corresponds to the occupied column of highest index. 

By definition of $c_{i}$, it is clear that $\rho_{i}$ has a rook in the cell $(i, c_{i})$. If $c_{i} = r+c$, there can be no outer corner in row $i$. If $c_{i}<r+c$, then $\rho_{i}$ has no rook in $(i, r+c)$. Then $(i, c_{i}+1)$ is an outer corner of $\lambda/\mu$ if and only if the rightmost cell of row $i$ occurs strictly to the left of the rightmost cell of row $i-1$. This is in turn equivalent to $c_{i}<c_{i-1}$. 

For the second part, let $i \in [r]$ be a row index and suppose $\lambda/\mu$ has an outer corner at $(i, k)$. By virtue of being $i$-extremal, the rook of $\rho_{i}$ in row $i$ occurs in column $k-1$, so $k \in I_{i}$. \textcolor{black}{Assume $R_{i}$ is an interval; if not, we are immediately done. If there exists an unoccupied column $j<k$ between two occupied columns of $\rho_{i}$, then $j, k \in I_{i}$, but $k-1 \notin I_{i}$, so $I_{i}$ is not cyclic. So assume all the occupied column indices of $\rho_{i}$ also form an interval.  Let $(p, q)$ be the bottommost rook of $\rho_{i}$ on $\lambda/\mu$. If $q = r+1$, then $p, k \in I_{i}$ and $k-1 \notin I_{i}$, so $I_{i}$ is not a cyclic interval. If $q>r+1$, then $q-1, k \in I_{i}$ but each of $q, \ldots , k-1 \notin I_{i}$, and hence $I_{i}$ is not an interval.} 
\end{proof}

The next lemma gathers further properties of Grassmann necklaces of rook matroids that we will use in the characterization theorem ahead. 

\begin{lemma}\label{lem:c_and_r_inqualities}
Let $\rookMat_{\lambda/\mu}$ be a rook matroid with Grassmann necklace $\mathcal{I} = (I_{1}, \ldots , I_{n})$, where the non-nesting rook placement corresponding to each $I_{i}$ is $\rho_{i}$. Then \begin{enumerate}
    \item If $j$ is a column index, then $c_{r_{j}}\geq j-1$. 
    \item If $i$ is a row index, then~$r_{c_{i}+2}-1 \leq i$.
    \item If $(t_{j}, s_{j})$ and $(t_{j+1}, s_{j+1})$ are successive rooks of $\rho_{i}$, then either $t_{j+1} = t_{j}+1$ or $s_{j+1}=s_{j}-1$.   
\end{enumerate}  
\end{lemma}

\begin{proof}
Let $\rho_{k}$ be the rook placement corresponding to $I_{k}$. Property (3) states that successive cells of $\rho_{i}$ either have adjacent rows or adjacent columns; this follows immediately from the construction of $i$-extremal rook placements. 

For (1), we need to show that column $j-1$ is either equal to, or occurs strictly to the left of, column $c_{r_{j}}$. If $j$ is a column index then, by definition, there is a rook $A$ of $\rho_{j}$ at $(r_{j}, j-1)$. Also, $\rho_{r_{j}}$ has a rook $B$, placed as far to the right as possible, at $(r_{j}, c_{r_{j}})$. Comparing the column indices of rooks $A$ and $B$ and noting that $B$ is $r_{j}$-extremal, (1) follows. 

For (2), we need to show that row $i$ is either equal to, or occurs strictly below, row $r_{c_{i}+2}-1$. There is a rook of $\rho_{i}$ at $(i, c_{i})$; move it horizontally one step to $(i, c_{i}+1)$ and call it rook $E$. (It will not matter for us if this cell lies in $\lambda/\mu$ or not.) Now, the rook placement $\rho_{c_{i}+2}$ has empty columns from $c_{i}+2$ to $r+c$, with the first rook occurring at $(r_{c_{i}+2}, c_{i}+1)$. By virtue of $(c_{i}+2)$-extremality, this rook is placed in the topmost cell of column $c_{i}+1$; call it rook $F$. Now rooks $E$ and $F$ both occur in column $c_{i}+1$ with the latter occurring in topmost cell of the column. It follows that the row index of rook $F$ must be less than or equal to that of rook $E$; this is precisely the statement in $(2)$. 
\end{proof}

In addition, we set the following notation for a Grassmann necklace $\mathcal{I}$, intended to model the outer and inner corners of the skew shape:
\begin{align}
    OC(\mathcal{I}) &= \{(i, c_{i}+1): i \in [2, n-k], c_{i}<c_{i-1}\}\label{eq:IC_definition},\\ IC(\mathcal{I}) &= \{(r_{j}-1, j-1): j \in [n-k+2, n-1], r_{j}>r_{j+1}, r_{j} \neq 1\}\label{eq:OC_definition}.
    \end{align}

Indeed, when $\mathcal{I}$ is the Grassmann necklace of a rook matroid we have: 
\begin{corollary}\label{cor:inner_corners-coincide}
    Let $\mathcal{I}$ be the Grassmann necklace of a rook matroid $\rookMat_{\lambda/\mu}$. Then $IC(\mathcal{I})$ and $OC(\mathcal{I})$ coincide with the sets of inner and outer corners of $\lambda/\mu$ respectively. 
\end{corollary}

\begin{proof}
    This is an immediate consequence of Lemma~\ref{lem:outer_corners}. 
\end{proof}

Recall the notation $R_{i}, C_{i}$ attached to a Grassmann necklace $\mathcal{I}$ from the beginning of this subsection. We will order their elements as \[ R_{i} = \{t_{1}< \ldots < t_{\ell} \}, \quad [n-k+1, n] \setminus C_{i} = S_{i} =  \{s_{\ell}<\ldots < s_{1}\}.\] 

We do so because when $\mathcal{I}$ is the Grassmann necklace of a rook matroid, then the non-nesting rook placement $\rho_{i}$ corresponding to $I_{i}$ is $\rho_{i} = \{(t_{i}, s_{i}): i = 1, \ldots ,\ell\}$. 

We use the above notation to state the following theorem which combines the previous lemmata to give necessary and sufficient conditions for a given positroid to be a rook matroid, answering a question of Lam~\cite{Lam2024PrivateCommunication}. 

\begin{theorem}\label{grassmann_necklace_characterization}
    Let $M$ be a loopless and coloopless positroid of rank $k$ on $[n]$. Let $\mathcal{I} = (I_{1}, \ldots , I_{n})$ be the Grassmann necklace of $M$. Then $M$ is a rook matroid if and only if all the following conditions hold:~\begin{enumerate}
        \item $I_{n-k+1} = [n-k+1, n]$,
         \item If $j \in [n-k+2, n]$ then $c_{r_{j}}\geq j-1$, 
    \item If $i \in [1, n-k]$ then~$r_{c_{i}+2}-1 \leq i$,
    \item If $i \in [n-k+2, n]$ and $t_{j}, t_{j+1} \in R_{i}$ and $t_{j}+1, \ldots , t_{j+1}-1 \notin R_{i}$, then $(t_{j+1}-1, s_{j+1}) \in IC(\mathcal{I})$,~and 
    \item If $i \in [1, n-k]$ and $s_{j}, s_{j+1} \in S_{i}$, and $s_{j}-1, \ldots s_{j+1}-1 \in C_{i}$, then~$(t_{j+1}, s_{j+1}+1) \in OC(\mathcal{I})$.
    \end{enumerate}
\end{theorem}

\begin{proof} 
    To see that conditions (1) to (5) are necessary, recall Proposition~\ref{prop:grassmann_of_rooks} that specifies that for a Grassmann necklace $\mathcal{I}$ of a rook matroid $\rookMat_{\lambda/\mu}$ each $I_{i}$ corresponds to the $i$-extremal non-nesting rook placement on $\lambda/\mu$. For (1), the $(n-k+1)$-extremal rook placement corresponds to the empty rook placement, which is represented by $[n-k+1, n]$. Conditions (2) and (3) above respectively correspond to points (1) and (2) of Lemma~\ref{lem:c_and_r_inqualities}. For (4), let $i$ be a column index and suppose $\rho_{i}$ has successive rooks at $(t_{j}, s_{j})$ and $(t_{j+1}, s_{j+1})$ such that the rows in between $t_{j}$ and $t_{j+1}$ are unoccupied. We want to show that $(t_{j+1}-1, s_{j+1}) \in IC(\mathcal{I})$, which by Corollary~\ref{cor:inner_corners-coincide}, is equivalent to showing that  $(t_{j+1}-1, s_{j+1})$ is an inner corner of $\lambda/\mu$. By Lemma~\ref{lem:c_and_r_inqualities}~(4), we have $s_{j+1} = s_{j}-1$. Clearly $(t_{j+1}-1, s_{j+1})$ does not lie in $\lambda/\mu$, since otherwise $\rho_{i}$ would have a rook there which by hypothesis, it does not. In contrast $(t_{j+1}, s_{j+1})$ lies in $\lambda /\mu$ since $\rho_{i}$ has a rook at that cell. The cell $(t_{j+1}-1, s_{j})$ also lies in $\lambda/\mu$, since $(t_{j}, s_{j})$ does and $\lambda/\mu$ is a skew shape. It follows that $(t_{j+1}-1, s_{j+1}) \in IC(\mathcal{I})$. A similar argument establishes (5) in the case of rook matroids. 

    Now suppose conditions (1) to (5) hold for a given Grassmann necklace~$\mathcal{I}$. We will define a skew shape from $\mathcal{I}$ and show that its extremal non-nesting rook placements coincide exactly with the sets in $\mathcal{I}$. Recall the definition of $IC(\mathcal{I})$ and $OC(\mathcal{I})$ from Equations~\ref{eq:IC_definition} and~\ref{eq:OC_definition}.  

\textbf{Claim 1}: If both $OC(\mathcal{I})$ and $IC(\mathcal{I})$ are empty, then each $I_{i}$ is a cyclic interval for $i=1, \ldots , n$.

If $OC(\mathcal{I})$ is empty then for all $i \in [2, n-k]$ we have $c_{i} =c_{i-1}$, which in turn implies that $c_{i} = c_{1}$. If $c_{1} < n$, it follows that $n \in I_{1}$. Since $I_{1} = (I_{n}\setminus \{n\}) \cup j$ for some $j \in [n]$, it follows that that $j = n$ and hence $I_{n}=I_{1}$. This implies that $n$ is a loop of the positroid, a contradiction. Thus $c_{i} = c_{1} = n$, for all $i \in [2, n-k]$. It follows similarly that if $IC(\mathcal{I})$ is empty, $r_{i} = r_{n} = 1$ for all $i \in [n-k+2, n]$. To see that each $I_{i}$ is a cyclic interval, note that $IC(\mathcal{I})$ and $OC(\mathcal{I})$ being empty together with conditions (4) and (5) imply that neither of $R_{i}$ or $C_{i}$ can have non-consecutive elements. Hence each $R_{i}$ and $C_{i}$ are intervals. \textcolor{black}{Then $I_{i} = R_{i} \cup C_{i}$ is a cyclic interval, since for $i \in [n-k+2, n]$, we have $n, 1 \in I_{i}$; for $i \in [1, n-2k]$, we have $C_{i} = \emptyset$; and finally, for $i \in [n-2k+1, n]$ we have $n-k, n-k+1 \in I_{i}$. Thus the claim holds.} 

Since each $I_{i}$ is a cyclic interval of size $k$, the given positroid must be the uniform matroid $U_{k, n}$ which is exactly equal to the rook matroid $\rookMat_{(n-k)^{k}}$; see Example~\ref{example1:rectangles}. Now consider the case when $IC(\mathcal{I})$ and $OC(\mathcal{I})$ are not simultaneously empty. 

\textbf{Claim 2:} The sets $IC(\mathcal{I})$ and $OC(\mathcal{I})$ respectively define the inner and outer corners of a skew shape~$\lambda/\mu$.

Label the rows of the prospective skew shape from $1$ to $n-k$ and the columns from $n-k+1$ to $n$. We need to show that $IC(\mathcal{I})$ and $OC(\mathcal{I})$ define valid inner and outer corners of a skew shape respectively. By definition, the sequence of second coordinates of the elements of $OC(\mathcal{I})$ and the sequence of first coordinates of $IC(\mathcal{I})$ are both strictly decreasing. \textcolor{black}{We need to show that no element of $OC(\mathcal{I})$ occurs North-West of an element of $IC(\mathcal{I})$. In symbols, this means that for any row index $i$ and column index $j$ such that $i\leq r_{j-1}$ we need $j-1 < c_{i}+1$. Since $(c_m)_m$ is a decreasing sequence, $c_{r_{j}}<c_{i}<c_{i}+1$. By condition (2) on $IC(\mathcal{I})$, we have $c_{r_{j}}\geq j-1$ for column indices $j$ and thus $j-1<c_{i}+1$. Similarly, if $i$ is a row index and $j$ is a column index such that $c_{i}+1<j-1$ then $r_{j}-1<i$: since $(r_{m})_m$ is a decreasing sequence, we have $r_{j}<r_{c_{i}+2}\leq i+1$, where the second inequality uses condition (3).} Further, conditions (4) and (5) ensure that two distinct Grassmann necklaces $\mathcal{I}$ define distinct collections $IC(\mathcal{I}), OC(\mathcal{I})$. Thus, $IC(\mathcal{I})$ and $OC(\mathcal{I})$ define valid inner and outer corners of a skew shape which is hence uniquely determined; let $\lambda/\mu$ be this skew shape. To finish, we need to prove the following claim:

\textbf{Claim 3:} If $J_{i}$ is the set corresponding to $\rho_{i}$, the $i$-extremal non-nesting rook placement on $\lambda/\mu$, then the two Grassmann necklaces must be equal:  $J_{i} = I_{i}$ for $i=1, \ldots , n$. 

By construction, $J_{n-k+1} = [n-k+1, n] = I_{n-k+1}$, where the second equality follows from condition (1). We will show that for $i \in [n-k+2, n]$, $I_{i} = J_{i}$ by induction. A similar argument using condition (5) instead of (4) shows that $I_{i} = J_{i}$ for $i \in [1, n-k]$, so we omit the proof of this case. Consider $I_{n-k+2}$, the base case; two possibilities arise. 

If $I_{n-k+2}$ is cyclic, it must be equal to $ [n-k+2, 1]$. Then by Lemma~\ref{lem:outer_corners}, $r_{n-k+2} = r_{n-k+3}$. Now, to show that $J_{n-k+2} = I_{n-k+2}$, it suffices to show that $1 \in J_{n-k+2}$. If this does not hold, then $J_{n-k+2} \cap [1, n-k] = \{j\}$, with $j \geq 2$. This implies that the first rook of $\rho_{n-k+2}$ is in row $j$, which implies that $\lambda/\mu$ has an inner corner at $(j-1, n-k+1)$. By construction of the skew shape, $(j-1, n-k+1) \in IC(\mathcal{I})$. In other words, $r_{n-k+2} > r_{n-k+3}$, a contradiction. Thus $I_{n-k+2} = J_{n-k+2} = [n-k+2, 1]$.  

Now suppose $I_{n-k+2}$ is not cyclic. Then $I_{n-k+2} = [n-k+2, n] \cup \{j_{1}\}$ and $J_{n-k+2} = [n-k+2, n] \cup \{j_{2}\}$ for some $j_{1}, j_{2} \geq 2$. \textcolor{black}{By definition, $r_{n-k+2} = j_{1}$.  Since $j_{2} \geq 2$, the lone rook of $\rho_{n-k+2}$ occurs at $(j_{2}, n-k+1)$. By extremality of $\rho_{n-k+2}$, this means that $\lambda/\mu$ must have an inner corner at $(j_{2}-1, n-k+1)$.  By definition of $IC(\mathcal{I})$, it follows that this corner is also equal to $(r_{n-k+2}-1, n-k+1)$, which in turn implies that $j_{1} = j_{2}$.} Thus $I_{n-k+2} = J_{n-k+2}$ in this scenario as well, establishing the base case. 

Now consider $I_{i}$ for $i \in [n-k+3, n]$ and suppose $I_{i-1} = J_{i-1}$. We will show that $I_{i} = J_{i}$. Once again, two cases arise: 

(i) $r_{i}<r_{i-1}$: Then $r_{i} \in I_{i}$, but $r_{i} \notin I_{i-1}$ since $r_{i-1} = \min I_{i-1}$. So $I_{i} = (I_{i-1} \setminus i-1) \cup r_{i}$ and hence it suffices to show that $r_{i} \in J_{i}$. Now $r_{i} \in J_{i}$ if and only if the first rook of $\rho_{i}$ occurs in row $r_{i}$ which in turn occurs if and only if $(r_{i}, i-1) \in \lambda/\mu$. If there is no inner corner in column $i-1$, and $\min J_{i} = 1$, this hold vacuously and we are done. \textcolor{black}{Now suppose there is no inner corner in column $i-1$, and let $2 \leq x = \min J_{i}$. Then $(x, i-1) \in \lambda/\mu$, since the first rook of $\rho_{i}$ lies there. Suppose $x < r_{i}$, then $\lambda/\mu$ must have some cell $(r_{i}, m)$ South-East of $(x, i-1)$, or else $r_{i}$ would be an empty row. By Fact~\ref{fact:skew_shape_characterization}, it follows that $(r_{i}, i-1) \in \lambda/\mu$ and we are done. So   assume $r_{i}<x$.} Since there is no inner corner in column $i-1$, the first cell of columns $i-1$ and $i$ occurs at the same row index, $x$. Then there must be an inner corner at $(x-1, j-1)$, for some column index $j>i$. By the definition of $IC(\mathcal{I})$, $x-1 = r_{j}-1$ which implies that $x = r_{j}$. However, since $j>i$, $x = r_{j} \leq r_{i} < x$, a contradiction. Thus, there must be an inner corner in column $i-1$, the corresponding row index of which is $r_{i}-1$. This immediately implies that $(r_{i}, i-1) \in \lambda/\mu$, and hence $I_{i} = J_{i}$. 

(ii) $r_{i} = r_{i-1}$. We consider two further subcases:

(a) $R_{i-1}$ is not an interval. Suppose $R_{i-1} = \{t_{1}<t_{2}<\ldots \}$. Since $R_{i-1}$ is not an interval, there exists least $j$ such that $t_{j}, t_{j+1} \in R_{i-1}$ and $t_{j}+1 \in [1, n-k] \setminus  R_{i-1}$. The rooks of $\rho_{i-1}$ in rows $t_{j}, t_{j+1}$ occur at $(t_{j}, s_{j})$ and $(t_{j+1},s_{j}-1)$ (by Lemma~\ref{lem:c_and_r_inqualities} (4)). Also, since $t_{j}, t_{j+1} \in R_{i-1}$ and $t_{j}+1, \ldots , t_{j+1}-1 \notin R_{i-1}$, by condition (4) it follows that $(t_{j+1}-1, s_{j+1}) \in IC(\mathcal{I})$ and is hence an inner corner of $\lambda/\mu$.  
Let $I_{i} \setminus I_{i-1} = \{b\}$. If $b> t_{j+1}$, then reasoning similarly as in the last sentence of the last paragraph, we would have an inner corner at $(t_{j+1}-1, s_{j})$, which is not possible since we already have an inner corner at $(t_{j+1}-1, s_{j+1})$.  So $t_{j} < b < t_{j+1}$. \textcolor{black}{(We cannot have $b=t_{j+1}$ since $t_{j+1} \in I_{i-1}$.)} We claim that $b = t_{j}+1 \in R_{i}$. If not, then $t_{j}, b \in R_{i}$ and $t_{j}+1, \ldots , b-1 \notin R_{i-1}$. By condition (4), it follows that $(b-1, s_{j})$ is an inner corner of $\lambda/\mu$, which is a contradiction to the fact that there is a cell of $\lambda/\mu$ above $(b-1, s_{j})$ (since $(t_{j}, s_{j}) \in \lambda/\mu$ and $b>t_{j}$). Thus $I_{i} = I_{i-1} \setminus (i-1) \cup (t_{j}+1)$. Since $J_{i}$ and $I_{i}$ differ up to a row index, we will be done if we can show that $t_{j}+1 \in J_{i}$. We claim that $\rho_{i}$ has a rook at $(t_{j}+1, s_{j})$. Since $s_{j}, s_{j}-1 \notin C_{i-1}$, neither element lies in $C_{i}$ either. Since $C_{i} = [n-k+1, n] \cap J_{i}$, columns $s_{j}-1$ and $s_{j}$ are occupied in $\rho_{i}$. The only way this is possible is if the rook of $\rho_{i-1}$ at $(t_{j+1}, s_{j}-1)$ maintains its position in $\rho_{i}$ while the rook of $\rho_{i-1}$ at $(t_{j}, s_{j})$ moves to $(t_{j}, s_{j}+1)$ in $\rho_{i}$. To ensure that column $s_{j}$ is occupied in $\rho_{i}$ while maintaining the non-nesting nature of the configuration, there must be a rook of $\rho_{i}$ at $(t_{j}+1, s_{j})$, which lies on $\lambda/\mu$, since $(t_{j}+1, s_{j}-1)$ is an inner corner. This implies that $t_{i}+1\in J_{i}$ and hence $I_{i} = J_{i}$, finishing the proof of Case (ii) (a). See Figure~\ref{fig:case_ii_a_eg} for an example of this case. 

\begin{figure}[ht!]
    \centering
    \begin{subfigure}{0.27\textwidth}
\centering
   \begin{tikzpicture}[inner sep=0in,outer sep=0in]
\node (n) {\begin{varwidth}{6cm}{
\ytableausetup{boxsize=1.25em}
  \begin{ytableau} \none & \none[] & \none[\scriptstyle{s_{j}-1}] & \none[\scriptstyle{s_{j}}] & \none[] & \none[]  \\ \none[\scriptstyle{t_{j}}] & \none & \none & \rook & \scalebox{0.5}{\blackrook} &  \\ \none[\scriptstyle{t_{j}+1}] & \none & \none & \scalebox{0.5}{\blackrook} &  &  \\ \none[\scriptstyle{b}] & \none & \none &  &  \\ \none[\scriptstyle{t_{j+1}}] & \none & \rook \scalebox{0.5}{\blackrook} &  &  \\ \none[] &  &  \\ \none[] &  \\ \end{ytableau}}\end{varwidth}};
\end{tikzpicture}
\caption{$\rho_{i-1}$ in white and $\rho_{i}$ in black.}
    \label{fig:case_ii_a_eg}
\end{subfigure}
\hspace{1cm}
\begin{subfigure}{0.27\textwidth}
\centering
\begin{tikzpicture}[inner sep=0in,outer sep=0in]
\node (n) {\begin{varwidth}{6cm}{
\ytableausetup{boxsize=1.25em}
 \begin{ytableau} \none & \none[m'] & \none[m] & \none[] & \none[] & \none[] & \none[]  \\ \none[] & \none &  & \rook & \scalebox{0.5}{\blackrook}  &  &  \\ \none[] &  & \rook & \scalebox{0.5}{\blackrook}  &  \\ \none[j] & \rook \scalebox{0.5}{\blackrook}  \\ \none[] &  \\ \end{ytableau}
 }\end{varwidth}};
\end{tikzpicture}
\caption{$\rho_{i-1}$ in white and $\rho_{i}$ in black.}
    \label{fig:case_ii_b_eg}
\end{subfigure}
\hspace{1cm}
\begin{subfigure}{0.27\textwidth}
\centering
\begin{tikzpicture}[inner sep=0in,outer sep=0in]
\node (n) {\begin{varwidth}{6cm}{
\ytableausetup{boxsize=1.25em}
 \begin{ytableau} \none & \none[\scriptstyle{s_{j}-1}] & \none[\scriptstyle{s_{j}}] & \none[] & \none[] & \none[]  \\ \none[\hspace{-0.2cm} \scriptstyle{r_{i} =r_{i-1}}] & \none & \rook  & \scalebox{0.5}{\blackrook} &  &  \\ \none[\scriptstyle{t_{j}}] & \rook & \scalebox{0.5}{\blackrook}  &  &  &  \\ \none[\hspace{-0.2cm} \scriptstyle{t_{j}+1}] &  &  &  \\ \none[\hspace{-0.2cm}\scriptstyle{t_{j+1}}] & \scalebox{0.5}{\blackrook} &  \\ \end{ytableau}
 }\end{varwidth}};
\end{tikzpicture}
\caption{$\rho_{i-1}$ in white and $\rho_{i}$ in black.}
    \label{fig:case_ii_c_eg}
\end{subfigure}
    \caption{In the proof of Theorem~\ref{grassmann_necklace_characterization}: (A) In Case (ii) (a), $b$ must equal $t_{j}+1$ and $t_{j}+1 \in J_{i}$ because if $\rho_{i-1}$ has rooks at $(t_{j}, s_{j})$ and $(t_{j+1}, s_{j}-1)$, then $\rho_{i}$ has a rook at $(t_{j}+1, s_{j})$. (B) In Case (ii) (b) subcase (I), the element $m' \notin C_{i-1}$ is the column index corresponding to the row containing the first rook of $\rho_{i-1}$ that is in the last cell of its row. (C) In Case (ii) (b) subcase (II), $t_{j+1}$ must equal $t_{j}+1$ or else column $s_{j}$ would have a cell followed by an inner corner, a contradiction.}
    \label{fig:case_ii_eg}
\end{figure}
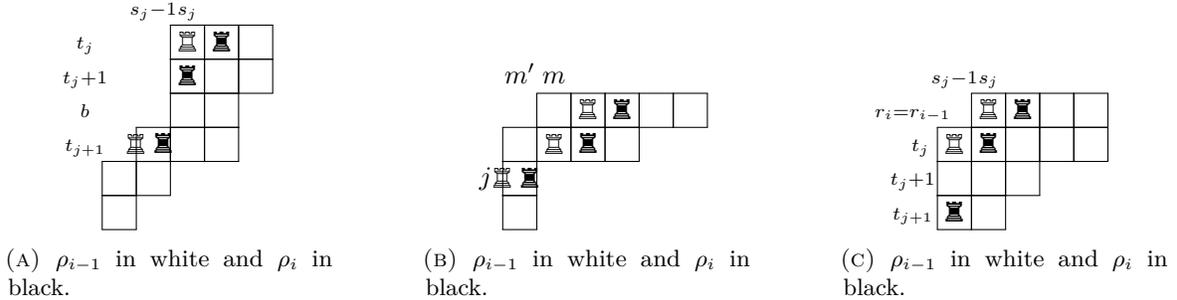

(b) $R_{i-1}$ is an interval, equal to $[r_{i-1}, t]$, say. Let $m = I_{i} \setminus I_{i-1}$. Two subcases arise: (I) $m \in C_{i}\setminus C_{i-1}$, and (II) $m \in R_{i} \setminus R_{i-1}$. 

Consider case (I): here, $R_{i}=R_{i-1}$, which means that $J_{i} \cap [1, n-k] = R_{i}$. Since $m \notin C_{i-1}$, $\rho_{i-1}$ has a rook in column $m$. Suppose $\rho_{i-1}$ has no rook that occurs in the last cell of its row. Then the rook in column $m$ must occur in the last row of the skew shape. Since $R_{i} = R_{i-1}$, $\rho_{i}$ is obtained from $\rho_{i-1}$ by moving every rook one cell to the right. In particular, occupied column $m$ in $\rho_{i-1}$ becomes unoccupied in $\rho_{i}$ (since the rook occurs in the last row), and hence $m \in J_{i}$. In the case when $\rho_{i-1}$ has a rook that occurs in the last cell of some row, let $j$ be the smallest such row index and let $m'$ be the corresponding column index. By condition (5) applied to $C_{i}$, there exists an outer corner in column $m'+1$. Since $\rho_{i-1}$ has a rook in row $m$, it must be the case that $m = m'+1$. Then since $R_{i} = R_{i-1}$, $\rho_{i}$ is obtained from $\rho_{i-1}$ by moving all the rooks of $\rho_{i-1}$ in rows $r_{i-1}$ to $j-1$ one cell to the right while maintaining the positions of the others. Since no rook of $\rho_{i-1}$ moves to column $m$ and the rook of $\rho_{i-1}$ in column $m$ moved to column $m+1$ in $\rho_{i}$, $m \in J_{i}$. This proves subcase (I) of Case (ii) (b). See Figure~\ref{fig:case_ii_b_eg} for an example of this case. 

Consider subcase (II): we first show that $R_{i}$ is also an interval. Let $R_{i-1} = \{t_{1}< \cdots < t_{j}\}$, where $t_{1} \coloneqq r_{i-1}$, $t_{j} \coloneqq t$ and $m = t_{j+1}$ so that $R_{i} = R_{i-1} \cup \{t_{j+1} \}= [r_{i-1}, t_{j}] \cup \{t_{j+1} \}$. By definition $S_{i} = S_{i-1} \cup \{i-1\}$, so if $S_{i} = \{s_{j+1}<s_{j}< \ldots < s_{1}\}$, then $S_{i-1} = \{s_{j+1}<s_{j}< \cdots < s_{2}\}$, where $s_{1} = i$. Now, suppose to the contrary that $t_{j+1} > t_{j}+1$. By condition (4), since $t_{j}, t_{j+1} \in R_{i}$ and are non-consecutive indices, $\lambda/\mu$ must have an inner corner at $(t_{j+1}-1, s_{j+1})$. By definition of $R_{i-1}, S_{i-1}$ and the fact that $I_{i-1} = J_{i-1}$,  the last rook of $\rho_{i-1}$ occurs at $(t_{j}, s_{j}-1)$, but since $t_{j}<t_{j+1}-1$, this is a contradiction to the existence of an inner corner at $(t_{j+1}-1, s_{j+1})$. Thus $R_{i}$ is also an interval. To finish, we need to show that $t_{j} + 1 \in J_{i}$. As noted above, the last rook of $\rho_{i-1}$ occurs at $(t_{j}, s_{j+1})$. Since $C_{i} = [n-k+1, n] \cap J_{i}$, $\rho_{i}$ will be obtained from $\rho_{i-1}$ by shifting every rook of $\rho_{i-1}$ one cell to the right and placing a new rook at $(t_{j}+1, s_{j+1})$. This implies that $t_{j}+1 \in J_{i}$, which proves subcase (II) of Case (ii) (b). See Figure~\ref{fig:case_ii_c_eg} for an example of this case. This completes the inductive step of $I_{i} = J_{i}$ for $i \in [n-k+2, n]$ and hence the proof.  
\end{proof}

We illustrate the theorem with a concrete example and non-example below. 

\begin{example}\label{exmp:characterization_theorem}
Let $n=11, k=5$. Consider the $(k, n)$-Grassmann necklace  $\mathcal{I}$ defined below.  \begin{align*}
I_{1} &= 12345, \quad \, \,\,\,\textcolor{red}{(c_{1} = 11)}  \qquad I_{7} = 789 \, 10 \, 11, \  \\ I_{2} &= 23456, \, \,\,\,  \textcolor{red} {\quad (c_{2} = 11)} \qquad I_{8} = 89\, 10\,11\,5, \quad \textcolor{red}{(r_{8} = 5)} \\ I_{3} &= 3456 \, 11, \quad \textcolor{red}{(c_{3} = 10)} \qquad I_{9} = 9\,  10 \, 11\, 35, \quad \textcolor{red}{(r_{9} = 3)}  \\
I_{4} &= 4569 \, 11, \quad \textcolor{red}{(c_{4} = 10)} \quad \, \, \, \,  I_{10} = 10 \, 11\, 135, \quad \textcolor{red}{(r_{10} = 1)}   \\ I_{5} &= 569\, 10\, 11, \, \, \textcolor{red}{(c_{5} = 8)} \qquad \, I_{11} = 11\, 1235. \quad \quad  \textcolor{red}{(r_{11} = 1)}  \\  I_{6} &= 689 \, 10 \, 11, \,\, \textcolor{red}{(c_{6} = 7)}
\end{align*}

In red, we have computed the corresponding $c_{i}$ or $r_{i}$ value depending on whether $i \in [1, n-k]$ or $[n-k+1, n]$. From here, we compute $IC(\mathcal{I})$ and $OC(\mathcal{I})$ as in (\ref{eq:IC_definition}) and (\ref{eq:OC_definition}). 
\[
IC(\mathcal{I}) = \{(4,7), (2, 8)\}, \quad OC(\mathcal{I}) = \{(3,11),(5,9), (6, 8)\}. 
\]

We now verify the conditions in Theorem~\ref{grassmann_necklace_characterization}. Condition (1) holds by inspection. Condition (2) ($c_{r_{j}} \geq j-1$ for $j \in [n-k+2, n]$) and Condition (3) ($r_{c_{i}+2}-1 \leq i$ for $i \in [1, n-k]$) hold by inspecting the pair of tables below. 

 \begin{center}
\begin{tabular}{ |p{1cm}|p{1cm}|p{1cm}|}
\hline
$j$ & $r_{j}$ & $c_{r_{j}}$  \\
\hline
8 & 5 & 8\\
\hline
9 & 3 & 10\\
\hline
10 & 1 & 11\\
\hline
11 & 1 & 11\\
\hline
\end{tabular}
\quad 
\begin{tabular}{ |p{1cm}|p{1cm}|p{1cm}|}
\hline
$i$ & $c_{i}$ & $r_{c_{i}+2}$  \\
\hline
1 & 11 & 0\\
\hline
2 & 11 & 0\\
\hline
3 & 10 & 0\\
\hline
4 & 10 & 0\\
\hline
5 & 8 & 1\\
\hline
6 & 7 & 3\\
\hline
\end{tabular}
\end{center}

For condition (4), the sets $I_{i}$ for $i \in [n-k+2, n]$ with non-consecutive row indices are $I_{9}, \ldots , I_{11}$. For each such $I_{i}$, we find the corresponding $R_{i} = \{t_{1}< \cdots < t_{\ell}\}$ and $S_{i} = \{s_{\ell}< \cdots < s_{1}\}$, then consider pairs $(t_{j}, s_{j})$ and $(t_{j+1}, s_{j+1})$ such that $t_{j+1} - t_{j} > 1$, and verify that $(t_{j+1}-1, s_{j+1}) \in IC(\mathcal{I})$. For $I_{9}$, we have $R_{11} =\{3,5\}$ and $S_{11} = \{7,8\}$ and $(3, 8), (5, 7)$ are pairs of cells with non-consecutive row indices. In this case, $(5-1, 7) \in IC(\mathcal{I})$. Similarly for $I_{10}$, $(1, 9)$ and $(3, 8)$ are pairs of cells with non-consecutive row indices, and $(2, 8) \in IC(\mathcal{I})$.

For condition (5), we can carry out a similar check on the sets $I_{i}$ for $i \in [1, n-k]$ and verify that the outer corner condition holds. Thus, $IC(\mathcal{I})$ and $OC(\mathcal{I})$ define valid inner and outer corners of a skew shape $\lambda/\mu$, which is drawn in Figure~\ref{fig:characterization_theorem_eg}. One can verify that the extremal non-nesting rook placements on $\lambda/\mu$ coincide with the Grassmann necklace above. 

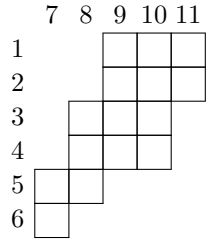
\begin{figure}[h]
    \centering
    \begin{tikzpicture}[inner sep=0in,outer sep=0in]
\node (n) {\begin{varwidth}{6cm}{
\ytableausetup{boxsize=1.25em}
\begin{ytableau} \none & \none[7] & \none[8] & \none[9] & \none[10] & \none[11]  \\ \none[1] & \none & \none &  &  &  \\ \none[2] & \none & \none &  &  &  \\ \none[3] & \none &  &  &  \\ \none[4] & \none  &  &  &  \\ \none[5] &  &  \\ \none[6] &  \\ \end{ytableau}}\end{varwidth}};
\end{tikzpicture}

\caption{$\lambda / \mu = 554421/221$ correponding to the Grassmann necklace in Example~\ref{exmp:characterization_theorem}. The inner and outer corners of the shape correspond to $IC(\mathcal{I})$ and $\mathcal{OC}(I)$ respectively.} 
    \label{fig:characterization_theorem_eg}
\end{figure}

\begin{example}
    Consider the same Grassmann necklace as in Example~\ref{exmp:characterization_theorem} but with $I_{11} = 11 \, 1245$ instead. The new Grassmann necklace has the same $IC(\mathcal{I})$ and $OC(\mathcal{I})$ as before. However, condition (4) of Theorem~\ref{grassmann_necklace_characterization} applied to $(2,8)$ and $(4, 7)$ requires that $(3, 7) \in IC(\mathcal{I})$, a contradiction. 
\end{example}

\end{example}

\section*{Acknowledgements}
I thank Thomas Lam for asking the question that motivated the main theorem in this work, and for fruitful discussions about positroids in general. I am very grateful to Irem Portakal and Akiyoshi Tsuchiya for discussions around this project, and to MPI Leipzig, where a portion of this work was carried out. I also thank Francesca Zaffalon for discussing her work on essential sets. I would  like to especially thank Katharina Jochemko for her comments and feedback. This work was supported by Wallenberg AI, Autonomous Systems and Software Program funded by the Knut and Alice Wallenberg Foundation.  

\bibliographystyle{alphaurl}
\bibliography{bibliography}

\end{document}